\documentclass[12pt,reqno,a4paper]{article}

\usepackage{lmodern}

\usepackage{graphicx}
\usepackage{amssymb}
\usepackage{amsmath}
\usepackage{amsthm}

\usepackage{url}
\usepackage{cite}
\usepackage{dsfont}
\usepackage{enumerate}

\usepackage[title]{appendix}
\allowdisplaybreaks[3]

\newcommand{\dd}{\,\mathrm{d}}
\renewcommand{\Tilde}{\widetilde}
\renewcommand{\Hat}{\widehat}
\renewcommand{\Bar}{\overline}

\newcommand{\R}{\mathbb{R}}

\newcommand{\N}{\mathbb{N}}

\newcommand{\NN}{\mathbb{N}}

\newcommand{\eps}{\varepsilon}

\newcommand{\cD}{\mathcal{D}}
\newcommand{\cH}{\mathcal{H}}

\newcommand{\U}{\mathcal{U}}
\newcommand{\V}{\mathcal{V}}

\newcommand{\supp}{\mathrm{supp\,}}

\newcommand{\vol}{\mathop{\mathrm{Vol}}\nolimits}

\newcommand{\spec}{\mathop{\mathrm{spec}}}
\newcommand{\ess}{\mathrm{ess}}

\newtheorem{theorem}{Theorem}[section]
\newtheorem{prop}[theorem]{Proposition}
\newtheorem{cor}[theorem]{Corollary}
\newtheorem{lemma}[theorem]{Lemma}
\theoremstyle{definition}
\newtheorem{rem}[theorem]{Remark}

\begin{document}

\title{\bf Asymptotics of Robin eigenvalues\\ on sharp infinite cones}

\author{Konstantin Pankrashkin\dag{} and Marco Vogel\ddag\\[\smallskipamount]
	\small Carl von Ossietzky Universit\"at Oldenburg,\\
	\small Institut f\"ur Mathematik,\\
	\small 26111 Oldenburg, Germany\\[\smallskipamount]
	\small \dag{} Corresponding author, ORCID: 0000-0003-1700-7295 \\
	\small E-Mail: \url{konstantin.pankrashkin@uol.de},\\
	\small Webpage: \url{http://uol.de/pankrashkin}\\[\smallskipamount]
	\small \ddag{} E-Mail: \url{marco.vogel@uol.de}
}

\date{}

\maketitle

%
%\address{Carl von Ossietzky Universit\"at Oldenburg, Institut f\"ur Mathematik,	Postfach 5634, 26046 Oldenburg, Germany}
%\email{konstantin.pankrashkin@uol.de}
%
%\author{Marco Vogel}
%\address{Carl von Ossietzky Universit\"at Oldenburg, Institut f\"ur Mathematik,	Postfach 5634, 26046 Oldenburg, Germany}
%\email{marco.vogel@uol.de}
%

\begin{abstract}
	Let $\omega\subset\mathbb{R}^n$ be a bounded domain with Lipschitz boundary.
	For $\varepsilon>0$ and $n\in\mathbb{N}$ consider the infinite cone
	\[
	\Omega_{\varepsilon}:=\big\{(x_1,x')\in (0,\infty)\times\mathbb{R}^n: x'\in\varepsilon x_1\omega\big\}\subset\mathbb{R}^{n+1}
	\]
	and the operator $Q_{\varepsilon}^{\alpha}$ acting as the Laplacian $u\mapsto-\Delta u$
	on $\Omega_{\varepsilon}$ with the Robin boundary condition
	$\partial_\nu u=\alpha u$ at $\partial\Omega_\varepsilon$, where $\partial_\nu$ is the outward normal derivative and $\alpha>0$.
	We look at the dependence of the eigenvalues of $Q_\eps^\alpha$ on the parameter $\eps$: this problem
	was previously addressed for $n=1$ only (in that case, the only admissible $\omega$ are finite intervals).	In the present work we consider arbitrary dimensions $n\ge2$ and arbitrarily shaped ``cross-sections'' $\omega$
	and look at the spectral asymptotics as 
	$\varepsilon$ becomes small, i.e. as the cone becomes ``sharp'' and collapses to a half-line. It turns out 	that the main term of the asymptotics of individual eigenvalues is determined by the single geometric quantity
	\[
		N_\omega:=\frac{\mathrm{Vol}_{n-1} \partial\omega}{\mathrm{Vol}_n \omega}.
	\]
	More precisely, for any fixed $j\in \mathbb{N}$ and $\alpha>0$ the $j$th eigenvalue $E_j(Q^\alpha_\varepsilon)$ of $Q^\alpha_\varepsilon$ exists for all sufficiently small $\eps>0$
	and satisfies
	\[
	E_j(Q^\alpha_\varepsilon)=-\frac{N_\omega^2\,\alpha^2}{(2j+n-2)^2\,\varepsilon^2}+O\left(\frac{1}{\varepsilon}\right) \text{ as $\varepsilon\to 0^+$.}
\]
The paper also covers some aspects of Sobolev spaces on infinite cones, which can be of independent interest.
\end{abstract}


\smallskip


\section{Introduction}\label{intro}
Let $\omega\subset\R^n$ be a bounded domain (connected open set) with Lipschitz boundary. 
For $\varepsilon>0$ consider the open set
\begin{equation*}
\Omega_{\eps}:=\big\{(x_1,x')\in (0,\infty)\times\R^n: \, x'\in \eps x_1\omega\big\}\subset\R^{n+1}.
%\label{def-veps}
\end{equation*}
Geometrically, the set $\Omega_\eps$ is an infinite cone in $\R^{n+1}$ such that
the intersection of $\Omega_\eps$ with the hyperplane $x_1=a$ gives the set $\eps a\omega$.
We are interested in some spectral properties of a Robin Laplacian on $\Omega_\eps$ as $\eps$ becomes small, i.e. when the cone becomes ``sharp'' and collapses to the half-line $(0,\infty)\times\{0\}$. Namely, for $\alpha>0$ denote by $Q_{\varepsilon}^{\alpha}$ the self-adjoint operator in $L^2(\Omega_{\varepsilon})$ generated by the closed, densely defined, symmetric bilinear form
\begin{align*}
q_{\eps}^{\alpha} (u,u)&=\int_{\Omega_{\eps}} |\nabla u|^2\dd x-  \alpha \int_{\partial\Omega_{\eps}} u^2\, \dd\sigma, \quad D(q_{\eps}^{\alpha})=H^1(\Omega_{\eps}),
\end{align*}
where $\dd\sigma$ stands for the $n$-dimensional Hausdorff measure. The semiboundedness and the closedness
are not completely obvious as $\Omega_\eps$ is unbounded and may have a non-Lipschitz singularity at the origin: we
discuss these aspects in detail in the appendices. Informally, the operator $Q_{\eps}^{\alpha}$ can be viewed as the positive Laplacian, $u\mapsto -\Delta u$, with the Robin boundary condition $\partial_{\nu}u=\alpha u$, where  $\partial_{\nu}$ is the outward normal derivative; we refer to \cite{bm,gm,gm2} for a discussion of various aspects related to the precise description
of the operator domain. Such operators are often referred
to as Robin Laplacians with negative parameters~\cite{bfk} due to the negative contribution of the boundary term in the bilinear form.  The cone $\Omega_\eps$ is invariant with respect to the dilations $x\mapsto tx$ for any $t>0$, and
standard arguments show the unitary equivalence
$Q_{\eps}^{\alpha}\simeq \alpha^{2}Q^1_{\eps}$. Hence, it will be convenient to consider $\alpha=1$ only and to study
the operator and the form
\[
Q_{\eps}:=Q_{\eps}^{1}, \quad
q_\eps:=q^1_\eps.
\]

For a review of spectral problems with Robin boundary conditions we refer to~\cite{bfk}.
In particular, the eigenvalues of Robin Laplacians on infinite cones play a central role in the strong coupling asymptotics of Robin eigenvalues on general domains.
Namely, if $\Omega$ is an open set in some large class and $T^{\Omega,\alpha}$ is the Robin Laplacian on $\Omega$
defined as the operator associated with the symmetric bilinear form
\[
t^{\Omega,\alpha}(u,u)=\int_\Omega|\nabla u|^2\dd x- \alpha \int_{\partial\Omega} u^2\, \dd\sigma, \quad u\in H^1(\Omega),
\]
then the lower edge $\Lambda_1(T^{\Omega,\alpha})$ of the spectrum of $T^{\Omega,\alpha}$ satisfies
\[
\Lambda_1(T^{\Omega,\alpha})=\alpha^2 \inf_{x\in \partial\Omega} \Lambda_1(T^{U_x,1}) +o(\alpha^2)\quad \text{as $\alpha\rightarrow+\infty$},
\]
where $T^{U_x,1}$ is the Robin Laplacian on the infinite tangent cone $U_x$ at $x\in\partial\Omega$.
We refer to~\cite{bp,levitin} for technical details and precise definitions and to \cite{HK,HKR,hp,KKR,kobp,pp15b} for a more precise eigenvalue analysis under more specific regularity
assumptions. The function $\alpha\mapsto \Lambda_1(T^{\Omega,\alpha})$ plays a role in the study of some non-linear equations as discussed in \cite{lacey}.
Eigenvalues and eigenfunctions of sharp cones can be used to produce counterexamples to spectral gap estimates \cite{kielty}.  In addition, such operators
attract some attention as examples of geometric ``long-range'' configurations producing an infinite discrete spectrum \cite{bpp,exmin,kp16}. Let us summarize the available spectral information for $Q_\eps$.

The essential spectrum of $Q_\eps$ depends in a non-trivial way on $\omega$ and $\eps$. If $\omega$ has smooth boundary, then  in virtue of \cite[Thm.~1]{kp16} the essential spectrum of $Q_\eps$ is $[-1,+\infty)$, as $\Omega_\eps$ is smooth outside the origin. For non-smooth $\omega$ the essential spectrum is determined through an iterative procedure and can look differently: see the detailed discussion in~\cite{bp}.

If $\omega$ is the unit ball centered at the origin of $\R^n$, then $Q_\eps$ is a round cone
whose lateral
surface  forms the constant angle $\theta:=\arctan\varepsilon$ with the central axis,
and the bottom of the spectrum of $Q_\eps$ is the eigenvalue
	\begin{equation}
		\label{eball}
	E_1(Q_\eps)=-\dfrac{1}{\sin^2\theta}\equiv -\dfrac{1+\eps^2}{\eps^2}
	\end{equation}
	with eigenfunction $\psi(x_1,x')=\exp(-x_1/\sin\theta)$. In fact, only $n=1$ and $n=2$ were considered explicitly,
	see e.g. \cite[Lem.~2.6]{levitin} and \cite[Prop.~4.2]{kl22}, but the constructions literally hold for arbitrary dimensions $n$.

The ~case $n=1$ ($\Omega_\eps$ is an infinite planar sector) was studied in detail in \cite{kp18}.
The only admissible sets $\omega$ are finite intervals, so without loss of generality we can take $\omega:=(-1,1)$.
In \cite{kp18} it was shown that the discrete spectrum of $Q_\eps$ is always finite, but  the number of eigenvalues grows unboundedly as $\eps$ becomes small, and for each  fixed $j\in\NN$ (we use the convention $0\notin\N$)
the $j$th eigenvalue
	$E_j(Q_\eps)$ behaves as
\begin{equation}
\label{ej2d}
E_j(Q_\eps)=-\frac{1}{(2j-1)^2\eps^2}+O(1) \text{  as $\eps\to 0^+$.}
\end{equation}
Some explicit formulas for eigenpairs of $Q_\eps$ in this particular case were obtained in~\cite{lyal}, but it is unclear
if the constructed family exhausts the whole discrete spectrum.
	
If $n\ge 2$, the discrete spectrum of $Q_\eps$ may be infinite. For example, if $n=2$
and $\omega$ is simply connected with smooth boundary, then the infiniteness of the discrete spectrum follows by \cite[Cor.~8]{kp16}, since the complement of $\Omega_\eps$ is not a convex set  (similar arguments apply in higher dimensions: we refer to \cite{kp16} for details). On the other hand, for polyhedral $\omega$ the discrete spectrum can be finite. For example, if one chooses $\omega$ in such a way that $\Omega_1$ is an isometric copy of $(0,\infty)^{n+1}$, then an easy analysis based on the separation
of variables method shows that the discrete spectrum of $Q_1$ consists of a single eigenvalue $-(n+1)$.
For $n=2$ and smooth $\omega$, the accumulation rate of eigenvalues at the bottom of the essential spectrum
was studied in \cite{bpp}. Furthermore, in \cite{kl22} it was shown that round infinite
cones maximize the first eigenvalue among all cones with the same perimeter of the spherical cross-section.
Various two-sided estimates for the bottom of the spectrum were obtained in \cite{levitin}. In particular, it was shown that the lowest eigenvalue can be computed explicitly if the spherical cross-section of $\Omega_\eps$ is a spherical polygon admitting an inscribed circle.

In the present work we complement the above results by computing the asymptotics of \emph{individual} eigenvalues of $Q_\eps$ for small $\eps$ in arbitrary dimensions and for arbitrary cross-sections $\omega$.
It turns out that the main term in the asymptotics depends on a single geometric constant $N_\omega$ given in \eqref{asymp} and, hence, it is rather insensitive to the regularity of $\omega$. Our result reads as follows:

\begin{theorem}\label{thm1}
Let $j\in \N$, then $Q_\eps$ has at least $j$ discrete eigenvalues below the bottom of the essential spectrum for all sufficiently small $\eps>0$, and its $j$th eigenvalue $E_j(Q_\eps)$  satisfies
	\begin{equation}
		  \label{asymp}
	E_j(Q_\eps)=-\frac{N_\omega^2}{(2j+n-2)^2\,\eps^2}+O\left(\frac{1}{\eps}\right) \text{ as } \eps\to 0^+,
	\quad 	N_\omega:=\dfrac{\vol_{n-1} \partial\omega}{\vol_n \omega}.
	\end{equation}
\end{theorem}
For $n=1$ and $\omega=(-1,1)$ one has $N_\omega=1$, and the result follows directly from \eqref{ej2d},
and all other intervals $\omega$ are easily included by applying suitable reparametrizations.
Hence, \emph{for the rest of the text	we explicitly assume $n\ge 2$}. Remark that if $\omega$ is a unit ball centered at the origin, then one has $N_\omega=n$, and the exact formula \eqref{eball}
has the form \eqref{asymp} with $j=1$ and a more accurate remainder estimate. Based on these observations
one may expect that the remainder estimate in \eqref{asymp} is not optimal.
We further remark that if the volume $\vol_n \omega$ or the surface area $\vol_{n-1} \partial\omega$ is fixed, then
the quantity $N_\omega$ is minimized by the ball due to the classical isoperimetric inequality.
Hence, the sharp cones $\Omega_\eps$ whose cross-section $\omega$ are balls maximize the main  term
in \eqref{asymp} among all sharp cones with cross-sections of the same volume or
surface area.

Our proof is variational and based on the min-max principle, and its main ingredient is a kind of asymptotic separation of the variables $x_1$ and $x'$, which is quite similar to \cite{kp18}, but the analysis in the $x'$-direction is much more involved and uses some coordinate transforms similar to~\cite{kov}. Various proof steps are explained in greater detail in Subsection~\ref{sec-iso peak} below.
We intentionally opted for the min-max approach as it allows for an elementary and self-containing proof of the main results concerning the eigenvalues. It should be noted that other more advanced approaches should be applied if one needs more information on eigenfunctions, for example, the method of matched asymptotic expansions \cite{mnp}. Nevertheless, in view of the expected amount of additional
technical work we believe that such an analysis should be the subject of separate study. We also remark that in the appendix we prove some results on Sobolev spaces on $\Omega_\eps$ (which is unbounded and may be non-Lipschitz) that are needed for the spectral analysis: this part of the text may be of its own interest.

\section{Preparations for the proof}  \label{sec-prelim}

\subsection{Min-max principle}\label{ssmm}

If $T$ is a lower semibounded, self-adjoint operator in an infinite-dimensional Hilbert space $\cH$. The spectrum and the essential spectrum of $T$ will be denoted by $\spec T$ and $\spec_\ess T$ respectively. Furthermore, denote $\Sigma:=\inf\spec_\ess T$ for $\spec_\ess T\ne \emptyset$ and $\Sigma:=+\infty$ otherwise. If $T$ has at least $j$ eigenvalues (counting multiplicities) in $(-\infty,\Sigma)$, then we denote by $E_j(T)$ its $j$th eigenvalue (when enumerated in the non-decreasing order and counted according to the multiplicities).  All operators we consider are real (i.e. map real-valued functions to real-valued functions), and we prefer to work with real Hilbert spaces in order to have shorter expressions. 

Let $t$ be the bilinear form for $T$, with domain $D(t)$, and let $D\subset D(t)$ be any dense subset (with respect to the scalar product induced by
$t$).
Consider the following ``variational eigenvalues''
\begin{equation*}
	\Lambda_j(T):=\inf_{\substack{V\subset D\\ \dim V=j}}\sup_{\substack{u\in V \\ u\neq 0}}\frac{t(u,u)}{\langle u,u\rangle_{\cH}}, 
\end{equation*}
which are independent of the choice of $D$. One easily sees that $j\mapsto \Lambda_j(T)$ is non-decreasing, and it is known \cite[Section XIII.1]{RS4} that only two cases are possible:
\begin{itemize}
	\item For all $j\in\mathbb{N}$ there holds $\Lambda_{j}(T)<\Sigma$. Then the spectrum of $T$ in $(-\infty,\Sigma)$ consists of infinitely
	many discrete eigenvalues $E_j(T)\equiv \Lambda_j(T)$ with $j\in\NN$.
	
	\item For some $N\in\NN\cup\{0\}$ there holds $\Lambda_{N+1}(T)\geq\Sigma$, while $\Lambda_j(T)<\Sigma$ for all $j\le N$.
	Then $T$ has exactly $N$ discrete eigenvalues in $(-\infty,\Sigma)$ and $E_j(T)=\Lambda_{j}(T)$ for $j\in\{1,\dots,N\}$, while $\Lambda_{j}(T)=\Sigma$ for all $j\geq N+1$.
\end{itemize}
In all cases there holds $\lim_{j\to\infty}\Lambda_j(T)=\Sigma$, and if for some $j\in\NN$ one has $\Lambda_{j}(T)<\Sigma$, then $E_j(T)=\Lambda_j(T)$.
In particular, if for some $j\in \NN$ one has the strict inequality $\Lambda_j(T)<\Lambda_{j+1}(T)$, then $E_j(T)=\Lambda_j(T)$.

\subsection{Robin Laplacian on $\omega$} 
\label{sec-1d}

Given $r\in\R$ denote by $B_{r}$ the self-adjoint operator in $L^2(\omega)$ generated by the closed symmetric bilinear form 
\begin{equation} \label{d-form} 
	b_{r}(f,f) =\int_{\omega} |\nabla f(t)|^2\, \dd t   -  r \int_{\partial\omega} f(t)^2\, \dd\tau(t),  \qquad  f\in H^1(\omega);
\end{equation} 
remark that $b_r$ is semibounded from below due to Proposition~\ref{prop-sobolev}.
Informally, the operator $B_r$  is the Laplacian $f\mapsto -\Delta f$ on $\omega$
with the Robin boundary condition $\partial_\nu f = r f$, with $\partial_\nu$ being the outward normal derivative. We will summarize some important spectral properties of $B_r$ as follows.

\begin{lemma} \label{lem-1}
	
	The following assertions hold true:
	
	\begin{enumerate}
		\item[(a)] For any $r\in\R$ the first eigenvalue $E_1(B_r)$ is simple, and the corresponding eigenfunction $\psi_r$ can be chosen strictly positive with $\|\psi_r\|_{L^2(\omega)}=1$.
	
		\item[(b)] The mappings $\R\ni r\mapsto E_1(B_r)\in \R$ and
		 $\R\ni r\mapsto \psi_r\in L^2(\omega)$ are $C^\infty$.

		\item[(c)] There exists $\varphi\in L^\infty(0,\infty)$ such that 
		$E_1(B_r) = -N_\omega r +r^2 \varphi(r)$ for all $r>0$ and $N_\omega$ as defined in \eqref{asymp}.

		\item[(d)] Let $E_2^N>0$ be the second eigenvalue of the Neumann Laplacian on $\omega$, then
		$\lim_{r\to 0}E_2(B_r)  = E_2^N$.

		\item[(e)] For any $r_0>0$ there exists $K>0$ such that 
		\begin{equation} \label{eq-kp}
			\int_{\omega} \big|\partial_r \psi_ r(y)\big|^2\dd y  \leq K \text{ for all  $r \in (0,r_0)$.}
		\end{equation} 
	\end{enumerate}	
\end{lemma} 

\begin{proof}
	Part (a) is proved for even more general Robin problems in~\cite[Sec.~4.2]{ateg}. Both  (b) and (d) follow from the fact that the operators $B_r$ form a type (B) analytic family with respect to $r$, see \cite[Ch.~7, \S 4]{kato}, and
	(e) is a direct consequence of (b). To prove (c) we remark first that there exists $C>0$ such that
	\begin{equation}
		 \label{eq-c1}
	-Cr^2\le E_1(B_r)\le 0 \text{ as $r\to+\infty$};
	\end{equation}
the lower bound is proved e.g. in~\cite[Corol.~2.2]{kov}, and the upper bound follows from $b_r(1,1)<0$ (which holds for all $r>0$) by the min-max principle. Furthermore, by Eq.~(4.16) in \cite{bfk} one has
\[
\dfrac{d}{dr}E_1(B_r)\big|_{r=0}=-N_\omega,
\]
and it follows that $E_1(B_r)=-N_\omega r+O(r^2)$ as $r\to 0^+$. By combining this asymptotics with \eqref{eq-c1} we arrive at the representation in (c).
\end{proof}

\subsection{One-dimensional model operators}\label{compop}
For a fixed $\lambda>0$  consider the symmetric differential operator in $L^2(0,\infty)$ given by
\begin{equation}
	C_c^\infty(0,\infty) \ni f \mapsto -f'' + \bigg(\frac{n^2-2n}{4s^2}-\frac{N_\omega}{\lambda s}\bigg) f
\end{equation}
and denote by $A_{\lambda}$ its Friedrichs extension. Note that $n^2-2n\ge 0$ due to $n\ge 2$.

In \cite[Chapter~8.3]{gitman} the spectrum of $A_\lambda$ was fully determined\footnote{For $n=2$ see p.~312 and for $n\geq 3$ see p.~294 in \cite{gitman}.}: the essential spectrum is $[0,+\infty)$ and the negative eigenvalues are simple and
are explicitly given by
\begin{equation}\label{1d EV}
	E_j(A_\lambda)=\dfrac{E_j( A_{1})}{\lambda^2}=-\frac{N_\omega^2}{(2j+n-2)^2\lambda^2},\qquad j\in\N,\quad \lambda>0.
\end{equation} 

In what follows we will need to work with truncated versions of $A_\lambda$. Namely, for $b>0$
we denote by $M_{\lambda,b}$ and $\Tilde M_{\lambda,b}$ the Friedrichs extensions in $L^2(0,b)$ 
and $L^2(b,\infty)$ of the operators $C_c^\infty(0,b)\ni f\mapsto A_\lambda f$ and $C_c^\infty(b,\infty)\ni f\mapsto A_\lambda f$ respectively. 

Note that, by construction, the form domain of $M_{\lambda,b}$ is contained in $H^1_0(0,b)$, which implies that
$M_{\lambda,b}$ has compact resolvent. We need to relate the eigenvalues of $M_{\lambda,b}$ to those of $A_\lambda$.
As the bilinear form of $A_\lambda$ extends that of $M_{\lambda,b}$, one has, due to the min-max principle,
\begin{equation}
	\label{eq-1dd}
	E_j(M_{\lambda,b})\ge E_j(A_\lambda) \text{ for any $b>0$, $\lambda>0$, $j\in\N$.}
\end{equation}
Let us now obtain an asymptotic upper bound for $E_j(M_{\lambda,b})$.

\begin{lemma}\label{lem7}
	Let $b>0$ and $j\in\NN$. Then there exist $K>0$ and $\eps_0>0$ such that
	\[
	E_j(M_{\eps,b}) \le E_j( A_\eps) +K \text{ for all } \eps\in(0,\eps_0).
	\]
\end{lemma}

\begin{proof}
	The proof is quite standard and uses a so-called IMS partition of unity \cite[Sec.~3.1]{bsim}.
	Let $\chi_1$ and $\chi_2$ be two smooth functions on $\R$ with $0\le\chi_1,\chi_2\le 1$, such that
	$\chi_1^2+\chi_2^2=1$, $\chi_1(s)=0$ for $s>\frac{3}{4}\, b$, $\chi_2(s)=0$ for $s<\frac{1}{2}\, b$.
	We set $K:=\|\chi_1'\|_\infty^2+\|\chi_2'\|_\infty^2$.
	An easy computation shows that for any $f\in C^\infty_c(0,\infty)$ there holds
	\begin{align*}
		\int_0^\infty |f'|^2 \dd s & =\int_0^\infty \big|(\chi_1 f)'\big|^2 \dd s +\int_0^\infty \big|(\chi_2 f)'\big|^2 \dd s
		-\int_0^\infty \Big(|\chi_1'|^2+|\chi_2'|^2\Big) f^2\, \dd s\\
		& \ge \int_0^\infty \big|(\chi_1 f)'\big|^2 \dd s +\int_0^\infty \big|(\chi_2 f)'\big|^2 \dd s
		- K \|f\|^2_{L^2(0,\infty)},
	\end{align*}
	which implies
	\begin{align*}
		\big\langle f, A_{\eps} f\big\rangle_{L^2(0,\infty)}+K \|f\|^2_{L^2(0,\infty)}&\ge \big\langle \chi_1 f, A_{\eps}(\chi_1 f)\big\rangle_{L^2(0,\infty)}+\big\langle 
		\chi_2 f, A_{ \eps}(\chi_2 f)\big\rangle_{L^2(0,\infty)}\\
		&\equiv \big\langle \chi_1 f, A_{\eps}(\chi_1 f)\big\rangle_{L^2(0,b)}+\big\langle 
		\chi_2 f, A_{ \eps}(\chi_2 f)\big\rangle_{L^2(\frac{b}{4},\infty)}.
	\end{align*}
	Using the identity $\|f\|^2_{L^2(0,\infty)}=\|\chi_1 f\|^2_{L^2(0,b)}+\|\chi_2 f\|^2_{L^2(\frac{b}{4},\infty)}$
	and the obvious inclusions $\chi_1 f\in C^\infty_c(0,b)$, $\chi_2 f\in C^\infty_c(\frac{b}{4},\infty)$,
	we apply the min-max principle as follows:
	\begin{equation}
		\label{eq-ineq00}
		\begin{aligned}
			E_j(A_{\eps})+K &= \inf_{\substack{S\subset C^\infty_c(0,\infty)\\ \dim S=j}} \sup_{\substack{f\in S\\ f\ne 0}} \dfrac{\langle f, A_\eps f\rangle+K\|f\|^2_{L^2(0,\infty)}}{\|f\|^2_{L^2(0,\infty)}}\\
			& \ge \inf_{\substack{S\subset C^\infty_c(0,\infty)\\ \dim S=j}} \sup_{\substack{f\in S\\ f\ne 0}} \dfrac{\big\langle \chi_1 f, A_{\eps}(\chi_1 f)\big\rangle_{L^2(0,b)}+\big\langle \chi_2 f, A_{\eps}(\chi_2 f)\big\rangle_{L^2(\frac{b}{4},\infty)}}{\|f\|^2_{L^2(0,\infty)}}\\
			& = \inf_{\substack{S\subset C^\infty_c(0,\infty)\\\dim S=j}} \sup_{\substack{f\in S \\ f\ne 0}} \dfrac{\big\langle \chi_1 f, A_{\eps}(\chi_1 f)\big\rangle_{L^2(0,b)}+\big\langle \chi_2 f, A_{\eps}(\chi_2 f)\big\rangle_{L^2(\frac{b}{4},\infty)}}{\|\chi_1 f\|^2_{L^2(0,b)}+\|\chi_2 f\|^2_{L^2(\frac{b}{4},\infty)}} \\
			& \ge \inf_{\substack{S\subset C^\infty_c(0,b)\oplus C^\infty_c(\frac{b}{4},\infty)\\ \dim S=j}}\,  \sup_{\substack{(f_1,f_2) \in S\\ (f_1,f_2)\ne 0}} \dfrac{\big\langle f_1 , A_{\eps}\, f_1\big\rangle_{L^2(0,b)}
				+\big\langle f_2 , A_{\eps}\, f_2\big\rangle_{L^2(\frac{b}{4},\infty)}}{\|f_1\|^2_{L^2(0,b)}+\|f_2\|^2_{L^2(\frac{b}{4},\infty)}}\\
			& =\Lambda_j\big(M_{\eps,b}  \oplus \Tilde M_{\eps, \frac{b}{4}}\big)\ge \min\big\{\Lambda_j(M_{\eps,b}),\inf\spec \Tilde M_{\eps, \frac{b}{4}}\big\}.
		\end{aligned}
	\end{equation}
For any $j\in\NN$ we have $\Lambda_j(M_{\eps,b})=E_j(M_{\eps,b})$. At the same time, for any function $f\in C_c^\infty(\frac{b}{4},\infty)$ one has
	\begin{align*}
		\langle f, \Tilde M_{\eps, \frac{b}{4}} f\rangle_{L^2(\frac{b}{4},\infty)}&=\int_{\frac{b}{4}}^\infty\Big[|f'|^2+
		\Big(\frac{n^2-2n}{4s^2}-\frac{N_\omega}{\eps s}\Big) f^2\Big]\dd s\\
		&\ge -\frac{N_\omega}{\eps}\int_{\frac{b}{4}}^\infty \dfrac{1}{s}\,f^2\dd s\ge  -\frac{4N_\omega}{b\eps}\|f\|^2_{L^2(\frac{b}{4},\infty)},
	\end{align*}
	which gives the lower bound $\inf\spec \Tilde M_{\eps, \frac{b}{4}}\ge  -\frac{4N_\omega}{b\eps}$. Due to \eqref{1d EV} we conclude that if $j\in\NN$ is fixed, then one can find some $\eps_0>0$
	such that for all $\eps\in(0,\eps_0)$ there holds $E_j(A_{\eps})+K<\inf\spec\Tilde M_{\eps, \frac{b}{4}}$.
	Then \eqref{eq-ineq00} implies $E_j(A_{\eps})+K\ge E_j(M_{\eps,b})$.
\end{proof}

\subsection{Scheme of the proof}\label{sec-iso peak}

We first remark that $\Omega_\eps$ is unbounded and, in general, not with Lipschitz boundary near $0$, and it does not satisfy the standard assumptions for trace theorems and other important assertions discussed in most books. These aspects are discussed in detail
in the appendices, and here we only cite the most important conclusions.

For an open interval $I\subset (0,\infty)$ we denote 
\begin{equation}
	\label{cinfi}
	C^\infty_I(\Bar \Omega_\eps):=\big\{ u\in C^\infty(\Bar \Omega_{\eps}): \text{ $\exists\, [b,c]\subset I$ such that }
	\text{$u(x)=0$ for $x_1\notin [b,c]$} \big\},
\end{equation}
in Appendix \ref{appa} we show:
\begin{lemma}\label{prop4dens}
	The subspace $ C^\infty_{(0,\infty)}(\Bar \Omega_\eps)$ is dense in $H^1(\Omega_\eps)$.
\end{lemma}

To continue we need suitable coordinates on $\partial\Omega_\eps$. Consider the diffeomorphism
\[
X:(0,\infty)\times\R^n\to (0,\infty)\times\R^n, \quad
X(s,t)=(s,\eps st), \quad (s,t)\equiv(s,t_1,t_2,\dots,t_n),
\]
then $\partial \Omega_\eps=X \big((0,\infty)\times \partial\omega\big) \cup\{0\}$. Remark that $\{0\}$ has zero $n$-dimensional Hausdorff measure and can be neglected in the integration over $\partial\Omega_\eps$.

By $\dd\sigma$ and $\dd\tau$ we will denote the integration with respect to the $n$- and $(n-1)$-dimensional Hausdorff measures, respectively. The following technical estimate will play an important role (see Appendix \ref{appb} for a detailed proof):

\begin{lemma}\label{lem6a}
	For any $\eps>0$, any measurable $v:\partial\Omega_\eps\to \R$ and $u:=v\circ X$ there holds
	\begin{multline*}
		\eps^{n-1}\int_0^\infty \int_{\partial\omega} s^{n-1}\big|u( s, t)\big| \dd\tau(t) \dd s\\
		\le
		\int_{\partial \Omega_\eps} |v|\dd\sigma
		\le
		\sqrt{1+R^2\eps^2}\eps^{n-1}\int_0^\infty \int_{\partial\omega} s^{n-1}\big|u( s, t)\big| \dd\tau(t) \dd s
	\end{multline*}
	with $R:=\sup_{t\in\omega}|t|$.
\end{lemma}

By combining Lemmas \ref{prop4dens} and \ref{lem6a}
we show (Proposition \ref{propb1}) that the trace of any function $u\in H^1(\Omega_\eps)$ on  $\partial\Omega_\eps$ is well-defined
and, as a result, that the bilinear form $q_\eps$ is closed and semibounded from below (Corollary \ref{cor-trace}),
which shows that $Q_\eps$ is really a lower semibounded self-adjoint operator. Its spectral study will be mainly based
on the min-max principle (Susection \ref{ssmm}) in which we take $D:= C^\infty_{(0,\infty)}(\Bar \Omega_\eps)$.

In Section \ref{sec-teps} we study a Robin Laplacian on a finite part of $\Omega_\eps$, and it is the most voluminous part of the analysis. Pick some $a>0$ (this value will remain fixed through the whole text), and denote
\begin{align*}
	V_{\eps}&:=\Omega_\eps\cap\{x_1<a\}\equiv \big\{(x_1,x')\in (0,a)\times\R^n:\, x'\in \eps x_1\omega \big\}\subset\R^{n+1},\\
	\partial_0 V_{\eps}&:=\partial \Omega_\eps\cap\{x_1<a\}\equiv \big\{(x_1,x')\in (0,a)\times\R^n: \, x'\in \eps x_1\partial\omega \big\}\subset \partial V_{\eps},    \\
	\Hat H^1_0(V_{\eps})&:=\text{the closure of $C^\infty_{(0,a)}(\Bar\Omega_\eps)$ in $H^1(V_\eps)$}.
\end{align*}
Let $T_{\eps}$ be the self-adjoint operator in $L^2(V_{\eps})$ associated with the symmetric bilinear form 
\begin{equation}
	\label{teps1}
	t_{\eps} (u,u)=\int_{V_{\eps}} |\nabla u|^2\dd x- \int_{\partial_0 V_{\eps}} u^2\, \dd\sigma, \quad D(t_{\eps})=\Hat H^1_0(V_{\eps}),
\end{equation}
then $T_\eps$ can be informally interpreted as the Laplacian in $V_\eps$ with the Robin boundary condition $\partial_\nu u=u$ on $\partial_0 V_\eps$ and the Dirichlet boundary condition on the remaining boundary $\partial V_\eps\setminus \partial_0 V_\eps$ (which corresponds to $x_1=a$). In view of Lemma~\ref{prop4dens} the variational eigenvalues of $T_\eps$ are defined by
\begin{equation} \label{ej-vp0}
	\Lambda_j(T_{\eps}) = \inf_{\substack{S\subset D_0(t_\eps) \\ {\rm dim\, } S =j} } \, \sup_{\substack{u\in S\\ u\neq 0}}\,  \frac{t_{\eps}(u,u)}{\quad \|u\|_{L^2(V_{\eps})}^2}\,,
	\quad
	D_0(t_\eps):=C^\infty_{(0,a)}(\Bar\Omega_\eps),
	\quad j\in\NN.
\end{equation}
Using  a suitable change of coordinates and the spectral analysis of $B_{r}$, the study of eigenvalues of $T_\eps$ with small $\eps$ is reduced to the truncated one-dimensional operators $M_{\eps',a}$ (with suitable $\eps'\sim \eps$) from Subsection~\ref{compop}. The main result of this reduction is given in Proposition~\ref{prop8}. The analysis
is in the spirit of the Born-Oppenheimer approximation, see e.g.~\cite[Part~3]{nray}, with
$M_{\eps',a}$ being an ``effective operator'', and it is essentially an adaptation of
the constructions of the earlier paper~\cite{kov} on Robin eigenvalues in domains with peaks.

The link between the analysis of the truncated operator $T_\eps$ and the initial operator $Q_\eps$ is justified in Section \ref{ssthm1}.
We show in Proposition~\ref{prop9} that the eigenvalues of $Q_\eps$ are close to those of $T_\eps$, which finishes the proof of Theorem~\ref{thm1}.

%%%%%%%%%%%%%%%%%%%%%%%%%%%%%%%%%%%%%%%%%%%%

\section{Spectral analysis near the vertex}\label{sec-teps}

In this section we study $\Lambda_j(T_\eps)$ with small $\eps$. The proof will be based on \eqref{ej-vp0}
and on a kind of asymptotic separation of variables.

\subsection{Change of variables} \label{sec-subs0}
One observes that 
\begin{gather*}
	V_{\eps} = X (\Pi), \quad \Pi = (0,a)\times \omega, \quad
X(s,t)=(s,\eps st), \quad (s,t)\equiv(s,t_1,t_2,\dots,t_n) \in \Pi.
\end{gather*}
This induces the unitary transform (change of variables)
\begin{equation}\label{unitaryU}
\U: L^2(V_{\eps})\to L^2(\Pi, \eps^n s^{n} \dd s\,\dd t),
\quad
\U\, u := u\circ X.
\end{equation} 
Consider the symmetric bi linear form $p_{\eps}$ in $L^2(\Pi, \eps^ns^{n} \dd s\,\dd t)$
given by 
\[
p_{\eps}(u,u) := t_{\eps} (\U^{-1} u,\U^{-1} u), \quad D(p_{\eps})=\U D(t_{\eps}).
\]
Due to the unitarity of $\U$ and Lemma~\ref{prop4dens}, the subspace
\begin{align*}
D_0(p_{\varepsilon})&:=\U\, D_0(t_{\eps})\\
&\equiv  \big\{u\in C^\infty(\overline{\Pi}):\, \exists\, [b,c]\subset (0,a) \text{ such that } u(s,t)=0 \text{ for } s\notin[b,c]\big\},
\end{align*}
is a core of $p_{\eps}$, and by \eqref{ej-vp0} one has
\begin{equation} \label{ej-vp2a}
\Lambda_j(T_{\eps}) = \inf_{\substack{S\subset D_0(p_{\eps}) \\ {\rm dim\, } S =j} } \, \sup_{\substack{u\in S\\ u\neq 0}}\,  \frac{p_{\eps}(u,u)}{\quad \|u\|_{L^2(\Pi, \eps^n s^{n}\dd s\,\dd t)}^2}.
\end{equation} 

Now we would like to obtain more convenient expressions for $p_\eps(u,u)$.
\begin{lemma}\label{lem5}
Denote
\begin{equation}
	\label{rrr}
	R:=\sup_{t\in\omega} |t|.
\end{equation}
	For any $v\in D_0(t_\eps)$ and $u:=\U v\in D_0(p_\eps)$,
\begin{multline*}
\eps^n\int_0^a \int_{\omega}
\Big[\big(1-nR\eps \big)\, |\partial_s u|^2 + \dfrac{1-(nR^2\eps^2+R\eps)}{\eps^2 s^2 }\, |\nabla_t u|^2\Big]s^n\dd t \dd s\\
\leq \int_{V_\eps} |\nabla v|^2\dd x
\leq \eps^n\int_0^a \int_{\omega} 
\Big[\big(1+nR\eps \big)\, |\partial_s u|^2 + \dfrac{1+(nR^2\eps^2+R\eps)}{\eps^2 s^2 }\, |\nabla_t u|^2\Big]s^n\dd t \dd s.
\end{multline*}
\end{lemma}

\begin{proof}
A standard computation shows that for any $u\in D_0(p_{\varepsilon})$ there holds
\begin{equation}
	 \label{eq-form1a}
\int_{V_\eps} |\nabla v|^2\dd x=\eps^n\int_0^a \int_{\omega}  \langle \nabla u, G\,  \nabla u \rangle_{\R^{n+1}}\,  s^{n}  \dd t \dd s
\end{equation}
where $G$ is the $(n+1)\times (n+1)$ matrix given by
\[
G= (D X^T \, DX)^{-1}\equiv \left(\, 
 \begin{matrix}
1+\eps^2|t|^2 &   \eps^2 s\,  t  \\[\smallskipamount]
 \eps^2 s\,t^T  &  \eps^2 s^{2}\, \mathds{1}
\end{matrix}\, \right)^{-1}
\]
with $DX$ being the Jacobi matrix of $X$ and $\mathds{1}$ being the $n\times n$ identity matrix.
One checks directly that
$G$  is a block matrix,
\begin{equation*} 
G= \left(\, 
 \begin{matrix}
1&  -\dfrac ts\,    \\[\medskipamount]
-\dfrac {t^T}{s}  &  C
\end{matrix}\, \right)\, 
\quad \text{with} 
\quad 
C_{jk} = \begin{cases}
\dfrac{1}{\eps^2s^2}+\dfrac{t_j^2}{s^2} &  \text{ if} \quad j=k,\\[\bigskipamount]
%&\\
\dfrac{t_j t_k}{s^2} &   \text{ if} \quad  j\neq k.
\end{cases}
\end{equation*}
We would like to estimate the term $\langle \nabla u,G\,  \nabla u \rangle_{\R^{n+1}}$ from above and from below using simpler expressions. One obtains
\begin{multline}
	\langle \nabla u, G\,  \nabla u \rangle_{\R^{n+1}}
	=|\partial_s u|^2+\dfrac{1}{\eps^2 s^2}|\nabla_t u|^2 \\	-\frac{2}{s}\sum_{k=1}^n t_k \,\partial_s u\,\partial_{t_k} u+\dfrac{1}{s^2}\sum_{j,k=1}^n t_j t_k \,\partial_{t_j} u\,\partial_{t_k} u.
\label{scalar prod}
\end{multline}
Using the standard inequality $2|xy|\le x^2+y^2$ and $|t_j|\le|t|<R$ we estimate
\begin{align*}
	\Big|\dfrac{2}{s}\sum_{k=1}^n t_k \,\partial_s u\, \partial_{t_k} u\Big|&\le R \eps  \sum_{k=1}^n \Big|2 \partial_s u \cdot
	\dfrac{\partial_{t_k} u}{\eps s}\Big|\\
	&\le R\eps  \sum_{k=1}^n \Big(|\partial_s u|^2+ \dfrac{|\partial_{t_k} u|^2}{\eps^2 s^2}\Big)
	=n R\eps |\partial_s u|^2+ \dfrac{R}{\eps s^2}\,|\nabla_t u|^2,\\
	\Big|\dfrac{1}{s^2}\sum_{j,k=1}^n t_j t_k \,\partial_{t_j} u\,\partial_{t_k} u\Big|&
	\le \dfrac{R^2}{s^2}\sum_{j,k=1}^n | \partial_{t_j} u\,\partial_{t_k} u|\\
	&\le \dfrac{R^2}{2s^2} \, \sum_{j,k=1}^n\big( |\partial_{t_j} u|^2+|\partial_{t_k} u|^2\big)=\dfrac{n R^2}{s^2}\,|\nabla_t u|^2.
\end{align*}
The substitution into \eqref{scalar prod} gives a two-sided estimate for $\langle \nabla u, G\,  \nabla u\rangle_{\R^{n+1}}$, and the substitution into \eqref{eq-form1a} gives the claim.
\end{proof}

By applying Lemmas \ref{lem5} and \ref{lem6a} to both summands of $t_\eps$ in \eqref{teps1}
and by adjusting various constants
we obtain the following two-sided estimate
written in a form adapted for the subsequent analysis:

\begin{prop}\label{prop7}
There exist $c>0$ and $\eps_0>0$, with $c\eps_0<1$, both independent of the choice of $a$, such that for any $\eps\in(0,\eps_0)$
 and any $u\in D_0(p_\eps)$ there holds
\begin{gather*}
p^-_\eps(u,u)\le p_\eps(u,u)\le p^+_\eps(u,u),\\
 \begin{aligned}
p^\pm_\eps(u,u)&:=(1\pm c\eps)\eps^n\int_0^a \int_{\omega} s^n
\Big[|\partial_s u|^2 + \dfrac{1}{\eps^2 s^2 }\, |\nabla_t u|^2\Big]\dd t \dd s\\
&\quad -\frac{1}{1\pm c\eps}\eps^{n-1}\int_0^a \int_{\partial\omega} s^{n-1}u^2 \dd\tau(t) \dd s.
\end{aligned}
\end{gather*}
In particular, by \eqref{ej-vp2a} it follows that
for each $j\in\NN$ and any $\eps\in(0,\eps_0)$
there holds
\[
\inf_{\substack{S\subset D_0(p_{\eps}) \\ {\rm dim\, } S =j} } \, \sup_{\substack{u\in S\\ u\neq 0}}\,  \frac{p^-_{\eps}(u,u)}{\quad \|u\|_{L^2(\Pi, \eps^n s^{n}\dd s\,\dd t)}^2}\le	\Lambda_j(T_{\eps}) \le \inf_{\substack{S\subset D_0(p_{\eps}) \\ {\rm dim\, } S =j} } \, \sup_{\substack{u\in S\\ u\neq 0}}\,  \frac{p^+_{\eps}(u,u)}{\quad \|u\|_{L^2(\Pi, \eps^n s^{n}\dd s\,\dd t)}^2}.
\]
\end{prop}

\subsection{Upper bound for the eigenvalues of $T_\eps$} \label{sec-upperb}
We are going to compare the eigenvalues of $T_{\eps}$ with those of the truncated one-dimensional operators $M_{\eps',a}$.

\begin{lemma}\label{lem6}
There exist $c,c',\eps_0>0$, with $c\eps_0<1$, such that for any $j\in \N$ and any $\eps\in(0,\eps_0)$
there holds $\Lambda_j(T_\eps)\le (1+c\eps)E_j(M_{(1+c\eps)^2\eps,a})+ c'$.
\end{lemma}

\begin{proof}
Take first $c$ and $\eps_0$ as in Proposition~\ref{prop7}. Define a unitary transform
\[
\V: L^2(\Pi)\to L^2(\Pi, \eps^n s^{n} \dd s\, \dd t), \quad
(\V\, u)(s,t) = \eps^{-\frac{n}{2}} s^{-\frac{n}{2}}\,  u(s,t),
\] and consider the symmetric bilinear form $r_\eps^+(u,u) := p_\eps^+(\V\, u,\V\, u)$.
One easily sees that for any $u\in D_0(r_\eps^+):=\V^{-1} D_0(p_\eps)\equiv D_0(p_\eps)$
there holds
\begin{align*}
r_\eps^+(u,u) &=(1+c\eps)\int_0^a \int_{\omega} \left( \Big(\partial_s u-\frac{n  u}{2s}\Big)^2
+ \frac{1}{\eps^2 s^2}\, |\nabla_t u|^2 \right) \, \dd t \dd s \\
& \quad - \dfrac{1}{(1+c\eps)\eps}\int_0^a  \frac{1}{s}\,   \int_{\partial\omega} u^2\,  \dd\tau(t) \dd s.
\end{align*}
The substitution $u\mapsto \V u$ into the upper bound of Proposition~\ref{prop7} shows that
\begin{equation*}
\Lambda_j(T_\eps) \le \inf_{\substack{S\subset D_0(r^+_\eps) \\ \dim S =j} } \, \sup_{\substack{u\in S\\ u\neq 0}}\,  \frac{r^+_\eps(u,u)}{\quad \|u\|_{L^2(\Pi)}^2}.
\end{equation*} 
Using the density, on the right-hand side one can replace $r^+_\eps$ and $D_0(r^+_\eps)$ by the closure $\overline{r^+_\eps}$
and any dense subset $D\subset D(\overline{r^+_\eps})$. By Lemma~\ref{lem-cyl2} we can take
\[
D=\big\{u\in H^1(\Pi): \text{ there exists $[b,c]\subset(0,a)$ such that $u(x)=0$ for $x_1\notin[b,c]$}\big\},
\]
and we keep the symbol $r^+_\eps$ for $\overline{r^+_\eps}$ on $D$, as it is given by the same expression.
Therefore,
\begin{equation}  \label{1-upperb0}
	\Lambda_j(T_\eps) \le \inf_{\substack{S\subset D\\ \dim S =j} } \, \sup_{\substack{u\in S\\ u\neq 0}}\,  \frac{r^+_\eps(u,u)}{\quad \|u\|_{L^2(\Pi)}^2},
\end{equation}

Then integration by parts shows that for $u\in D$ one has
\begin{equation} \label{per-partes}
\int_\omega\int_0^a \frac{u\partial_s u}{s}\,    \dd s = \int_\omega\int_0^a \frac{u^2}{2s^2}\, \dd s, 
\end{equation}
which implies  
\begin{equation}
	  \label{reps1}
\begin{aligned}
r_\eps^+(u,u) & =  (1+c\eps)\int_0^a \int_{\omega} \Bigg(  \Big(|\partial_s u|^2 +\frac{n^2-2n}{4 s^2}\, u^2 \Big)+ \frac{1}{\eps^2 s^{2}}\, |\nabla_t u|^2 \Bigg) \dd t \dd s\\
& \quad - \dfrac{1}{(1+c\eps)\eps}\int_0^a  \frac{1}{s}\,   \int_{\partial\omega} u^2\,  \dd\tau \dd s\\
&=(1+c\eps)\bigg[\int_0^a \int_{\omega}  \Big(|\partial_s u|^2 +\frac{n^2-2n}{4 s^2}\, u^2 \Big)\dd t \dd s\\
&\quad + \int_0^a \frac{1}{\eps^2 s^2} \Big\{
\int_\omega |\nabla_t u|^2  \dd t - \frac{\eps s}{(1+c\eps)^2} \int_{\partial\omega} u^2\dd\tau
\Big\}\dd s\bigg].
\end{aligned}
\end{equation}
Note that the functional in the curly brackets is the bilinear form $b_{\eps\rho(s,\eps)}$ as defined in Subsection \ref{sec-1d} with $\rho(s,\eps) =s (1+c\eps)^{-2}$.
Let $\psi\equiv \psi_{\eps \rho(s,\eps)}$  be the positive normalized eigenfunction of $B_{\eps\rho(s,\eps)}$
for $E_1(B_{\eps\rho(s,\eps)})$. By Lemma \ref{lem-1}, for any $\eps>0$
the map $\R\ni s\mapsto \psi_{\eps\rho(s,\eps)}\in L^2(\omega)$ is $C^\infty$. If $f\in C^\infty_c(0,a)$, then
also $\R\ni s\mapsto f(s)\psi_{\eps\rho(s,\eps)}\in L^2(\omega)$ is $C^\infty$, and the derivative (which is smooth and with compact support)
coincides with the weak derivative in $\Pi$ with respect to $s$. It follows that the function $(s,t)\mapsto f(s)\psi_{\eps\rho(s,\eps)}(t)$
belongs to the above subspace $D$. Moreover, if $S \subset C_c^\infty(0,a)$ is a $j$-dimensional subspace, then
\begin{equation*} %\label{eq-S}
\Tilde S = \big\{ u:\Pi\to \R:\, u(s,t)= f(s)\, \psi_{\eps\rho(s,\eps)}(t), \, f\in S\big\} .
\end{equation*} 
is a $j$-dimensional subspace of $D$.
For any $u\in \Tilde S$  one has $\|u\|_{L^2(\Pi)} = \|f\|_{L^2(0,a)}$ by Fubini's theorem and
\[
\int_{\omega} |\nabla_t u|^2 \, \dd t - \frac{\eps s}{(1+c\eps)^2}\, \int_{\partial\omega} u^2\,  \dd\tau  = E_1(B_{\eps\rho(s,\eps)}) f(s)^2
\]
due to the spectral theorem. Furthermore,
\begin{align*}
 \int_0^a\int_{\omega} |\partial_s u|^2\dd t\dd s&=
 \int_0^a \int_{\omega} \big|f'(s) \psi_{\eps\rho(s,\eps)}(t)+f(s)\partial_s \psi_{\eps\rho(s,\eps)}(t)\big|^2\dd t\dd s\\
 &=\int_0^a \int_{\omega}
 \Big[f'(s)^2 \psi_{\eps\rho(s,\eps)}(t)^2+f(s)^2 |\partial_s \psi_{\eps\rho(s,\eps)}(t)|^2\\
 &\quad +f(s)f'(s) \cdot 2 \psi_{\eps\rho(s,\eps)}(t)\,\partial_s \psi_{\eps\rho(s,\eps)}(t)\Big]\dd t\dd s,
\end{align*}
while
\[
\int_{\omega} 2 \psi_{\eps\rho(s,\eps)}(t)\partial_s \psi_{\eps\rho(s,\eps)}(t)\dd t
=\int_{\omega} \partial_s \big| \psi_{\eps\rho(s,\eps)}(t)\big|^2\dd t=\partial_s \|\psi_{\eps\rho(s,\eps)}\|^2_{L^2(\omega)}=\partial_s1=0.
\]
Therefore,
\begin{multline*} 
 \int_0^a\int_{\omega}\Big( |\partial_s u|^2 +\frac{n^2 -2n}{4 s^2}\, u^2 \Big) \dd t \dd s
= \int_0^a \Big[  |f'|^2 + \Big(\frac{n^2-2n}{4s^2} + \int_{\omega} \!\! |\partial_s \psi_{\eps \rho(s,\eps)}|^2 \dd t\Big) f^2 \Big] \dd s.
\end{multline*}

The substitution into \eqref{reps1} shows that for any $u\in \Tilde S$ there holds
\[
r^+_\eps(u,u)= (1+c\eps)\int_0^a
\Big[  |f'|^2 + \Big(\frac{n^2-2n}{4s^2} + \int_{\omega} \!\! |\partial_s \psi_{\eps \rho(s,\eps)}|^2 \dd t+\frac{E_1(B_{\eps\rho(s,\eps)})}{\eps^2 s^2}\Big) f^2 \Big] \dd s.
\]
By the estimate \eqref{eq-kp} in Lemma \ref{lem-1} we can control the term with $\partial_s \psi$. Namely, for $s\in(0,a)$ and $\eps\in(0,\eps_0)$ the values of $\eps\rho(s,\eps)$ are contained in some bounded interval, and then one can find some $K>0$ such that
\begin{align*} %\label{eq-kp2} 
\int_{\omega} \big|\partial_s \psi_{\eps \rho(s,\eps)}(t)\big|^2\, \dd t&
=\eps^2\int_{\omega} \Big( \partial_\rho \psi_{\rho}(t) \big|_{\rho= \eps\rho(s,\eps)} \, \dfrac{\partial \rho(s,\eps)}{\partial s}\Big)^2 \, \dd t\\
&=\dfrac{\eps^2}{(1+c\eps)^4}\int_{\omega} \Big( \partial_\rho \psi_{\rho}(t) \big|_{\rho= \eps \rho(s,\eps)}\Big)^2 \, \dd t\le K\eps^2.
\end{align*}
Hence, for all $u\in\Tilde S$ and $\eps\in(0,\eps_0)$ one has
\[
r_\eps^+(u,u)  \leq (1+c\eps)  \int_0^a \left[ |f'|^2 +\left(\frac{n^2-2n}{4s^2} +K\eps^2 + \frac{E_1(B_{\eps\rho(s,\eps)})}{\eps^2s^{2 }} \right) f^2  \right] \dd s.
\]
Now we apply Lemma \ref{lem-1}(a,c) to the eigenvalue $E_1(B_{\eps\rho(s,\eps)})$: there exists $c_0>0$ such that for all $s>0$ and $\eps>0$ there holds
\begin{align*}
\frac{E_1(B_{\eps\rho(s,\eps)})}{\eps^2 s^{2}}
&\leq \frac{-N_\omega\, \eps\, \rho(s,\eps) +c_0\, \eps^2 \rho^2(s,\eps) }{\eps^2 s^{2}}=-\frac{N_\omega}{\, \eps\, s}\cdot \dfrac{1}{(1+c\eps)^2} + \frac{c_0}{(1+c\eps)^4}.
\end{align*}
Hence, for all $u\in\Tilde S$ and $\eps\in(0,\eps_0)$ we obtain
\begin{align*} \label{b+upperb}
\frac{r^+_\eps(u,u)}{ \quad \|u\|_{L^2(\Pi)}^2}  &\ \leq\   \frac{(1+c\eps) \displaystyle\int_0^a \left[ |f'|^2 +\left(\frac{n^2-2n}{4s^2} -\frac{N_\omega}{\, \eps\, s}\cdot \dfrac{1}{(1+c\eps)^2} \right) f^2  \right] \dd s}{\|f\|_{L^{2}(0,a)}^2} \\
&\qquad  + \frac{c_0}{(1+c\eps)^3} +(1+c\eps)K\eps^2\\
&=(1+c\eps) \dfrac{\langle f,M_{(1+c\eps)^2\eps,a}f\rangle_{L^2(0,a)}}{\|f\|_{L^{2}(0,a)}^2}
+ \frac{c_0}{(1+c\eps)^3} +(1+c\eps)K\eps^2\\
&\le (1+c\eps) \dfrac{\langle f,M_{(1+c\eps)^2\eps,a}f\rangle_{L^2(0,a)}}{\|f\|_{L^{2}(0,a)}^2} +c'
\text{ for } c':=c_0+(1+c\eps_0)^2K\eps_0^2.
\end{align*}
The constants $c,c',\eps_0$ are independent of $j$ and $S$. By \eqref{1-upperb0},
for any $j\in \NN$ and $\eps\in(0,\eps_0)$ there holds
\begin{align*}
\Lambda_j(T_\eps)&\le 
\inf_{\substack{S\subset D \\ \dim  S =j} } \, \sup_{\substack{u\in S\\ u\neq 0}}\,  \frac{r^+_\eps(u,u)}{\quad \|u\|_{L^2(\Pi)}^2}
\le
\inf_{\substack{S\subset C_c^\infty(0,a) \\ {\rm dim\, } S =j} } \, \sup_{\substack{u\in \Tilde S\\ u\neq 0}}\,  \frac{r^+_\eps(u,u)}{\quad \|u\|_{L^2(\Pi)}^2}\\
&\le (1+c\eps) \inf_{\substack{S\subset C_c^\infty(0,a) \\ \dim S =j} } \, \sup_{\substack{f\in  S\\ f\neq 0}}\dfrac{\langle f,M_{(1+c\eps)^2\eps,a}f\rangle_{L^2(0,a)}}{\|f\|_{L^{2}(0,a)}^2}
+c'\\
&=(1+c\eps)E_j(M_{(1+c\eps)^2\eps,a})+ c'.\qedhere
\end{align*}
\end{proof}

\begin{cor}\label{lem7a}
For any $j\in\NN$ there exist $k>0$ and $\eps_0>0$ such that
\[
\Lambda_j(T_\eps)\le -\dfrac{N_\omega^2}{(2j+n-2)^2\eps^2}+\dfrac{k}{\eps} \text{ for all $\eps\in(0,\eps_0)$.}
\]
\end{cor}

\begin{proof}
By  Lemma~\ref{lem7} for any fixed $c>0$ and $j\in\NN$ we can choose $K'>0$ such that
	\[
	E_j(M_{(1+c\eps)^2\eps,a})\le E_j(A_{(1+c\eps)^2\eps})+K'\equiv -\dfrac{N_\omega^2}{(2j+n-2)^2 \eps^2} \cdot  \Big(\frac{1}{1+c\eps}\Big)^4+K'
	\]
	if $\eps$ is small enough. The substitution into  Lemma~\ref{lem6} gives
	the result.
\end{proof}

\subsection{Lower bound for the eigenvalues of $T_\eps$} 

The lower bound for the eigenvalues of $T_\eps$ is also obtained using a comparison with the operators $M_{\eps',a}$
but requires more work.

\begin{lemma}\label{lem9}
Let $j\in\NN$, then there exist $\eps_0>0$ and  $k'>0$ such that 
\[
\Lambda_j(T_{\eps})\ge -\dfrac{N_\omega^2}{(2j+n-2)^2\eps^2}-\dfrac{k'}{\eps} \text{ for all $\eps\in(0,\eps_0)$.}
\]
\end{lemma}

\begin{proof}
Take $c$ and $\eps_0$ as in Proposition~\ref{prop7}. For $\eps\in(0,\eps_0)$ consider again the unitary transform $\V: L^2(\Pi)\to L^2(\Pi,\eps^n s^{n} \dd s\, \dd t)$, $(\V\, u)(s,t) = \eps^{-\frac{n}{2}} s^{-\frac{n}{2}}\,  u(s,t)$,
and the symmetric bilinear form $r_\eps^-(u,u) := p_\eps^-(\V\, u,\V\, u)$. The
reparametri\-zation $u\mapsto\V u$ in the lower bound of Proposition~\ref{prop7} leads to
\begin{equation}
  \label{low-eq00}
\Lambda_j(T_\eps)\ge  \inf_{\substack{S\subset D_0(r^-_\eps) \\ \dim S =j} } \, \sup_{\substack{u\in S\\ u\neq 0}}\,  \frac{r^-_\eps(u,u)}{\|u\|_{L^2(\Pi)}^2}, \quad D_0(r^-_\eps):=\V^{-1} D_0(p_{\eps})\equiv D_0(p_{\eps}).
\end{equation}
The substitution of $\V u$ into $p_\eps^-$ and the partial integration \eqref{per-partes} show that
\begin{equation}
   \label{eq-b}
\begin{aligned}
r_\eps^-(u,u) & =  (1-c\eps)\int_0^a \int_{\omega} \Bigg(  \Big(|\partial_s u|^2 +\frac{n^2-2n}{4 s^2}\, u^2 \Big)+ \frac{1}{\eps^2 s^{2}}\, |\nabla_t u|^2 \Bigg) \dd t \dd s\\
& \quad - \dfrac{1}{(1-c\eps)\eps}\int_0^a  \frac{1}{s}\,   \int_{\partial\omega} u^2\,  \dd\tau \dd s\\
&=(1-c\eps)\bigg[\int_0^a \int_{\omega}  \Big(|\partial_s u|^2 +\frac{n^2-2n}{4 s^2}\, u^2 \Big)\dd t \dd s\\
&\quad + \int_0^a \frac{1}{\eps^2 s^2} \Big\{
\int_\omega |\nabla_t u|^2  \dd t - \frac{\eps s}{(1-c\eps)^2} \int_{\partial\omega} u^2\dd\tau
\Big\}\dd s\bigg],\\
\rho(s,\eps)&:= \frac{s}{(1-c\eps)^2}\in (0,m), \quad m:=\frac{a}{(1-c \eps_0)^2}, \quad  \eps\in(0,\eps_0).
\end{aligned}
\end{equation}
The expression in the curly brackets is the bilinear form $b_{\eps\rho(s,\eps)}$ for the Robin Laplacian $B_{\eps\rho(s,\eps)}$ on $\omega$ as discussed in Subsection \ref{sec-1d}. Denote by $\psi_{\eps\rho(s,\eps)}$ the positive eigenfunction for $E_1(B_{\eps\rho(s,\eps)})$ with $\|\psi_{\eps\rho(s,\eps)}\|_{L^2(\omega)}=1$,
then $s\mapsto \psi_{\eps\rho(s,\eps)}$ is $C^\infty$ by Lemma~\ref{lem-1}. We decompose each $u\in D_0(r^-_\eps)$ as
\[
u = v + w  \ \text{ with } \ 
v(s,t)=  \psi_{\eps\rho(s,\eps)}(t) \, f(s), \quad  f(s) := \int_{\omega} u(s,t)\, \psi_{\eps\rho(s,\eps)}(t)\, \dd t. 
\]
By construction we have $f\in C^\infty_c(0,a)$ and, furthermore,
\begin{gather} \label{orth}
\int_{\omega} w(s,t)\, \psi_{\eps\rho(s,\eps)}(t)\, \dd t = 0 \quad \text{ for any $s\in (0,a)$,}\\
\|f\|_{L^2(0,a)}=\|v\|_{L^2(\Pi)}, \quad
\|f\|^2_{L^2(0,a)} + \|w\|^2_{L^2(\Pi)}=\|u\|^2_{L^2(\Pi)}.  \label{eq-normuw}
\end{gather}
The spectral theorem applied to $B_{\eps\rho(s,\eps)}$ implies that for any $u\in D_0(r^-_\eps)$ there holds
\begin{multline}
   \label{eq-utw}
\int_{\omega} |\nabla_t u(s,t)|^2 \, \dd t - \eps\rho(s,\eps)  \int_{\partial\omega} u(s,t)^2\,  \dd\tau(t)\\
\ge E_1(B_{\eps\rho(s,\eps)})\, f(s)^2 +E_2(B_{\eps\rho(s,\eps)})\, \int_{\omega} w(s,t)^2\, \dd t \, .
\end{multline}
By Lemma \ref{lem-1}(c) one can find a constant $c_1>0$ such that
\[
E_1(B_{x})= -N_\omega\, x+O(x^2)> -\dfrac{N_\omega\, x}{1-c_1 x} \quad \text{for all sufficiently small $x>0$.}
\]
We have $\eps \rho (s,\eps)\in[0,m\eps_0]$. By adjusting the value of $\eps_0$ we conclude that there exists $c_2>0$ such that for all $\eps\in(0,\eps_0)$ and $s\in(0,a)$ one has
\begin{align*}
\dfrac{E_1(B_{\eps \rho(s,\eps)})}{\eps^2 s^{2}}&\ge -\dfrac{N_\omega\, \eps \rho(s,\eps)}{\eps^2 s^{2}\big(1-c_1\eps \rho(s,\eps)\big)}= -\dfrac{N_\omega\, \eps s}{(1-c\eps)^2\eps^2 s^{2}\Big(1- \dfrac{c_1\eps s\mathstrut}{(1-c\eps)^2}\Big)}\\
&=-\dfrac{N_\omega}{\eps s\big( (1-c\eps)^2- c_1\eps s\big)}
\ge -\dfrac{N_\omega}{\eps s(1-c_2 \eps)}.
\end{align*}
%
%
%\ge -\dfrac{n}{\eps(1-c_1 \eps) s}.
%\]
By Lemma \ref{lem-1}(d) we can find $C_0>0$ such that $E_2(B_{x})= E_2^N+ o(1)\ge C_0$ for small $x>0$. Hence, if $\eps_0$ is sufficiently small, $s\in (0,a)$ and $\eps\in(0,\eps_0)$, then \eqref{eq-utw} implies
\[
\frac{1}{\eps^2 s^2}\left(\int_{\omega} |\nabla_t u|^2 \, \dd t - \rho(s,\eps) \int_{\partial\omega} u^2\,  \dd\tau \right) \, \geq \, -\dfrac{N_\omega\, }{\eps s(1-c_2 \eps) } f(s)^2 +\frac{C_{0}}{\eps^2 s^{2}}\, \int_{\omega} w^2\, \dd t \, ,
\]
which is valid for all $u\in D_0(r^-_\eps)$. The substitution of the last inequality into \eqref{eq-b} shows that for all $\eps\in(0,\eps_0)$ and $u\in D_0(r^-_\eps)$ there holds
\begin{align*}
r^-_{\eps}(u,u)  & \geq  (1-c\eps) \int_0^a \int_{\omega} \left( |\partial_s u|^2+\frac{n^2-2n}{4 s^2} \, u^2 \right) \dd t \dd s\\
&\quad  +  (1-c\eps)C_0 \int_0^a \int_{\omega}  \frac{w^2}{\eps^2 s^{2}}\, \dd t \dd s-
N_\omega\frac{1-c\eps}{1-c_2\eps} \int_0^a \dfrac{f^2}{\eps s} \, \dd s.
\end{align*}
To have a simpler expression choose a suitable $k>c$ and adjust $\eps_0$ so that for all $\eps\in(0,\eps_0)$ and $u\in D_0(r^-_\eps)$,
\begin{equation}
	\label{lowerb-b}
\begin{aligned}
	r^-_{\eps}(u,u)  & \geq  (1-k\eps) \int_0^a \int_{\omega} \left( |\partial_s u|^2+\frac{n^2-2n}{4 s^2} \, u^2 \right) \dd t \dd s\\
	&\quad  +  \dfrac{C_0}{2} \int_0^a \int_{\omega}  \frac{w^2}{\eps^2 s^{2}}\, \dd t \dd s-
	\dfrac{N_\omega}{1-k\eps} \int_0^a \dfrac{f^2}{\eps s} \, \dd s.
\end{aligned}
\end{equation}
For the sake of brevity we will denote 
\[
\psi := \psi_{\eps\rho(s,\eps)},
\quad
\psi_s:=\partial_s \psi, \quad v_s:=\partial_s v, \quad w_s:=\partial_s w\, .
\]
Let us study the first integral on the right hand side of \eqref{lowerb-b}. Using the orthogonality relations
\eqref{orth} we obtain
\begin{multline}
     \label{us}
\int_0^a\int_{\omega} \left( |\partial_s u|^2+\frac{n^2-2n}{4 s^2} \, u^2 \right) \dd t \dd s   = \int_0^a\!\! \int_{\omega} \Big( v_s^2+\frac{n^2-2n}{4 s^2} \, v^2 \Big) \dd t \dd s\\
+\int_0^a\!\! \int_{\omega} \left( w_s^2+\frac{n^2-2n}{4 s^2} \, w^2 \right) \dd t \dd s + 2 \int_0^a\!\! \int_{\omega} v_s\, w_s\, \dd t\, \dd s.
\end{multline}
We have 
\begin{align*}
\int_0^a \int_{\omega}  v_s^2\dd t\dd s&=\int_0^a \int_{\omega} |f'\psi+f\psi_s|^2\dd t\dd s\\
&=\int_0^a \int_{\omega} \big(|f'|^2\psi^2+f^2|\psi_s|^2 +2 ff' \psi_s \psi\big)\dd t\dd s.
\end{align*}
Due to the normalization of $\psi$,
\[
\int_{\omega} 2\psi \psi_s \dd t=\partial_s \int_{\omega} \psi^2\dd t=\partial_s 1=0,
\]
therefore,
\[
\int_0^a \int_{\omega}  v_s^2\dd t\dd s=\int_0^a \big(|f'|^2 + \|\psi_s\|^2_{L^2(\omega)} |f|^2\big)\dd s
\ge \int_0^a |f'|^2 \dd s
\]
and, consequently,
\begin{equation}
\int_0^a \int_{\omega} \Big( v_s^2+\frac{n^2-2n}{4 s^2} \, v^2 \Big) \dd t \dd s  \geq  \int_0^a \left[ \, |f'|^2 + \frac{n^2-2n}{4s^2}\, f^2 \right] \dd s. \label{vs-lowerb}
\end{equation}

In order to estimate the two last terms in \eqref{us} we note that  
\begin{equation} \label{crossed}
 2\int_0^a \int_{\omega} v_s\, w_s\, \dd t \dd s  =  2\int_0^a \int_{\omega} f'\, \psi\, w_s \, \dd t\dd s
 +2\int_0^a f \int_{\omega} \psi_s\, w_s\, \dd t\, \dd s
\end{equation}
and that, in view of \eqref{orth},
\[
\int_{\omega} (\psi w_s+\psi_s w)\dd t=\partial_s \int_{\omega} \psi w\dd t=\partial_s 0=0,
\quad
\int_{\omega} \psi\, w_s \, \dd t=-\int_{\omega}  \, \psi_s  w\, \dd t. 
\] 
Hence, using $|2xy|\le x^2+y^2$,
\begin{multline}
\Big | 2\int_0^a \int_{\omega} f'\, \psi\, w_s \, \dd t\dd s \Big |  = \Big | 2\int_0^a \int_{\omega} f'\, \psi_s\,w\,   \dd t\,\dd s\, \Big | \\
\leq \int_0^a \int_{\omega} \big(|f'|^2     \psi_s^2 + w^2\big)\, \dd t\dd s
=\int_0^a \|\psi_s\|^2_{L^2(\omega)} |f'|^2\dd s + \|w\|_{L^2(\Pi)}^2. \label{UB-est1}
\end{multline}
Similarly,
\begin{multline}
\Big |  2 \int_0^a f \int_{\omega} \psi_s\, w_s \dd t \dd s \Big |
=\Big |  2 \int_0^a \int_{\omega} f\cdot \dfrac{1}{\sqrt \eps}\psi_s\, \sqrt \eps w_s \dd t \dd s \Big |\\
\leq\int_0^a   \int_{\omega}  \Big(f^2\cdot \frac{1}{\eps} \psi_s^2 +\eps w_s^2\Big)\dd t\dd s
=\dfrac{1}{\eps}\int_0^a  f^2 \|\psi_s\|^2_{L^2(\omega)}\dd s + \eps \|w_s\|_{L^2(\Pi)}^2. \label{UB-est2}
\end{multline}
Now we represent
\[
\|\psi_s\|^2_{L^2(\omega)}=\int_{\omega} |\partial_s \psi_{\eps\rho(s,\eps)}(t)|^2\dd t
=\dfrac{\eps^2}{(1-c\eps)^4}\int_{\omega} \Big(\partial_\rho \psi_{\rho}(t)\Big|_{\rho=\eps\rho(s,\eps)}\Big)^2\dd t .
\]
As $\eps\rho(s,\eps)\in(0,\eps_0m)$ for all $s\in(0,a)$ and $\eps\in(0,\eps_0)$ we can use the estimate \eqref{eq-kp} of Lemma~\ref{lem-1}: there exists $K>0$ such that for all $\eps\in(0,\eps_0)$ and $ s\in(0,a)$ one has $\|\psi_s\|^2_{L^2(\omega)} \leq \, K\eps^2\le\eps$ (assuming that $\eps_0$ is sufficiently small). We now use the  obtained estimate in \eqref{UB-est1} and \eqref{UB-est2}, which gives
\begin{align*}
\Big | \int_0^a \int_{\omega} f'\, \psi\, w_s \, \dd t\dd s \Big | &\leq \eps \int_0^a |f'|^2\, \dd s + \|w\|_{L^2(\Pi)}^2,\\
\Big |   \int_0^a f \int_{\omega} \psi_s\, w_s\, \dd t\, \dd s      \, \Big |  \ &\leq K\eps \int_0^a f^2\, \dd s + \eps\|w_s\|_{L^2(\Pi)}^2.
\end{align*}
The substitution of these two inequalities into \eqref{crossed} gives
\[
\Big |   \int_0^a \int_{\omega} v_s\, w_s\, \dd t \dd s   \, \Big| \, \leq \,  \eps \int_0^a |f'|^2\, \dd s + K\eps \int_0^a f^2\, \dd s + \|w\|_{L^2(\Pi)}^2  + \eps\|w_s\|_{L^2(\Pi)}^2. 
\]
We now use the last obtained inequality and \eqref{vs-lowerb} in \eqref{us}, which gives
\begin{multline*} 
\int_0^a \int_{\omega} \left( |\partial_s u|^2+\frac{n^2-2n}{4 s^2} \, u^2 \right) \dd t \dd s 
 \geq \,   \int_0^a \left[ \, (1-\eps)\, |f'|^2 + \Big(\frac{n^2-2n}{4s^2} - K\eps \Big) f^2 \right] \dd s \\
 +  \int_0^a \int_{\omega}\left[ \, (1-\eps)\, w_s^2 + \Big(\frac{n^2-2n}{4s^2} - 1 \Big) w^2 \right] \dd t \dd s\, ,
\end{multline*}
for all $\eps\in(0,\eps_0)$ and $u\in D_0(r^-_\eps)$.
Using this lower bound in \eqref{lowerb-b} one arrives at
\begin{align*}
r^-_\eps(u,u)  & \ge   (1-k\eps)\int_0^a \left[ \, (1-\eps)\, |f'|^2 + \left(\frac{n^2-2n}{4s^2} - K\eps \right) f^2 \right] \dd s \\
& \quad +  (1-k\eps)\int_0^a \int_{\omega}\left[ \, (1-\eps) w_s^2 + \left(\frac{n^2-2n}{4s^2} - 1 \right) w^2 \right] \dd t \dd s\\
& \quad + \dfrac{C_0}{2}\int_0^a \int_{\omega}  \frac{w^2}{\eps^2 s^{2}}\, \, \dd t \dd s
-\dfrac{N_\omega}{1-k\eps}\int_0^a \frac{f^2}{\eps s}\dd s\\
&\ge 
(1-k\eps)\int_0^a \left[ \, (1-\eps)\, |f'|^2 + \left(\frac{n^2-2n}{4s^2} - K\eps \right) f^2 \right] \dd s\\
&\quad
-\dfrac{N_\omega}{1-k\eps}\int_0^a \frac{f^2}{\eps s}\dd s -\|w\|^2_{L^2(\Pi)}.
\end{align*}
By taking sufficiently large $b>k$ and $c'>1$ and a smaller value of $\eps_0$ one deduces from
the last inequality the simpler lower bound
\begin{align*}
	r^-_\eps(u,u)&\ge (1-b\eps)\int_0^a \Big[
	|f'|^2+\Big(\frac{n^2-2n}{4s^2}- \dfrac{N_\omega}{(1-b\eps)^2\eps s}\Big) f^2
	\Big]\dd s\\
	&\quad	-c'\|f\|^2_{L^2(0,a)}-c'\|w\|^2_{L^2(\Pi)}.
\end{align*}
Using the norm equality \eqref{eq-normuw} this is equivalent to
\begin{equation}
   \label{eq-reps0}
 \begin{aligned}
r^-_\eps(u,u) +c' \|u\|^2_{L^{2}(\Pi)}&\ge
(1-b\eps)\int_0^a \Big[
|f'|^2+\Big(\frac{n^2-2n}{4s^2}- \dfrac{N_\omega}{(1-b\eps)^2\eps s}\Big) f^2
\Big]\dd s\\
&\equiv (1-b\eps)\langle f, M_{(1-b\eps)^2\eps,a}f\rangle_{L^2(0,a)}.
\end{aligned}
\end{equation}

By the norm equality \eqref{eq-normuw}, the map $u\mapsto(f,w)$ uniquely extends to a unitary map $\Psi:L^2(\Pi)\to L^2(0,a)\oplus \cH$,
where $\cH$ is some closed subspace of $L^2(\Pi)$.
Let $h_\eps$ be the symmetric bilinear form in $L^2(0,a)\oplus \cH$ defined as the closure of
the form
\[
C_c^\infty(0,a)\times \cH\ni (f,w) \mapsto \int_0^a \left[ \, \, |f'|^2 + \left(\frac{n^2-2n}{4s^2} - \frac{N_\omega}{(1-b\eps)^2 \eps s} \right) f^2 \right]\dd s,
\]
then the corresponding self-adjoint operator in $L^2(0,a)\oplus \cH$ is $H_\eps=M_{(1-b\eps)^2 \eps,a}\oplus 0$.
The inequality \eqref{eq-reps0} reads as
$r^-_\eps(u,u)+ c'\|u\|^2_{L^2(\Pi)}\ge (1-b\eps)\,  h_\eps(\Psi u,\Psi u)$ for all $u\in D_0(r^-_\eps)$,
and the lower bound \eqref{low-eq00} for $\Lambda_j(T_\eps)$ implies that for any $j\in\NN$ and any $\eps\in(0,\eps_0)$
there holds
\begin{equation}
\begin{aligned}
	\Lambda_j(T_\eps)+c'&\ge  \inf_{\substack{S\subset D_0(r^-_\eps) \\ \dim S =j} } \sup_{\substack{u\in S\\ u\neq 0}}\,  \frac{r^-_\eps(u,u)+c'\|u\|^2_{L^2(\Pi)}}{\|u\|_{L^2(\Pi)}^2}\\
	&\ge \inf_{\substack{S\subset D_0(r^-_\eps) \\ \dim S =j} } \sup_{\substack{u\in S\\ u\neq 0}}\,  \frac{ (1-b\eps) h_\eps(\Psi u,\Psi u)}{\|\Psi u\|^2_{L^2(0,a) \oplus\cH}}\\
	&\ge (1-b\eps)\inf_{\substack{S\subset D(h_\eps)\\ \dim S=j }} \sup_{\substack{v\in S\\ v\neq 0}}\,  \frac{ h_\eps(v,v)}{\|v\|^2_{L^2(0,a) \oplus\cH}}=(1-b\eps)\Lambda_j(H_\eps).
\end{aligned}
   \label{eqlok10}
\end{equation}
By Lemma \ref{lem7}, for some $K_0>0$ and for all sufficiently small $\eps>0$ we have
\[
E_j(M_{(1-b\eps)^2\eps,a})\le E_j(A_{(1-b\eps)^2\eps})+K_0=-\dfrac{N_\omega^2}{(2j+n-2)^2 (1-b\eps)^4\eps^2} +K_0<0,
\]
hence, $\Lambda_j(H_\eps)=\Lambda_j(M_{(1-b\eps)^2\eps,a}\oplus 0)=E_j(M_{(1-b\eps)^2\eps,a})$, and it follows by \eqref{eqlok10}
that $\Lambda_j(T_\eps)+c'\ge (1-b\eps)E_j(M_{(1-b\eps)^2\eps,a})$.
By \eqref{eq-1dd} we have $E_j(M_{(1-b\eps)^2\eps,a})\ge E_j(A_{(1-b\eps)^2\eps})$, therefore,
\begin{align*}
\Lambda_j(T_\eps)&\ge (1-b\eps) E_j(A_{(1-b\eps)^2\eps})-c'\\
&=-\dfrac{N_\omega^2}{(2j+n-2)^2 (1-b\eps)^3\eps^2}-c'\ge -\dfrac{N_\omega^2}{(2j+n-2)^2 \eps^2}-\dfrac{k'}{\eps}
\end{align*}
for a suitably chosen $k'>0$ and all sufficiently small $\eps>0$.
\end{proof}

By combining Corollary \ref{lem7a} and  Lemma \ref{lem9} we obtain the main result of the section:
\begin{prop}\label{prop8}
	For any $j\in\NN$ there holds
	\[
	\Lambda_j(T_\eps)=-\dfrac{N_\omega^2}{(2j+n-2)^2 \eps^2} +O\Big(\dfrac{1}{\eps}\Big) \text{ as $\eps\to 0^+$.}
	\]
\end{prop}

\section{End of proof of Theorem \ref{thm1}}\label{ssthm1}

Note that the right-hand side of the asymptotics in Proposition~\ref{prop8} corresponds to the sought asymptotics for $E_j(Q_\eps)$ in Theorem \ref{thm1}. In order to conclude the proof of Theorem \ref{thm1} it remains to show that the eigenvalues of $Q_\eps$ and $T_\eps$
with the same numbers are close to each other. This will be done in several steps.

\begin{lemma}\label{lem10}
For any $j\in\NN$ and $\eps>0$ the inequality $\Lambda_j(Q_\eps)\leq \Lambda_{j}(T_\eps)$ holds.
\end{lemma}
\begin{proof}
Let $J:L^2(V_\eps)\to L^2(\Omega_\eps)$	be the operator of extension by zero, then $J$ is a linear isometry
with $JD(t_\eps)\subset D(q_\eps)$ and with $q_\eps(Ju,Ju)=t_\eps(u,u)$ for all $u\in D(t_\eps)$, and the result follows directly by the min-max principle.
\end{proof}

Recall that the subspaces $C^\infty_I(\Bar\Omega_\eps)$ were defined in \eqref{cinfi}.
For $b>0$ denote
\begin{align*}
	\Tilde V_{\eps,b}&:=\Omega_\eps\cap\{x_1>b\}\equiv \big\{(x_1,x')\in (b,\infty)\times\R^n:\, x' \in  \eps x_1\omega \big\}\subset\R^{n+1},\\
\partial_0 \Tilde V_{\eps,b}&:=\partial\Omega_\eps\cap\{x_1>b\}\equiv \big\{(x_1,x')\in (b,\infty)\times\R^n:\,  x'\in \eps x_1 \partial\omega\big\}\subset \partial \Tilde V_{\eps,b},  \\
\Hat H^1_0(\Tilde V_{\eps,b})&=\text{the closure of $C^\infty_{(b,\infty)}(\Bar\Omega_\eps)$ in $H^1(\Tilde V_{\eps,b})$}
\end{align*}
and let $\Tilde T_{\eps,b}$ be the self-adjoint operator in $L^2(\Tilde V_{\eps,b})$ defined by its symmetric bilinear form
\[
	\Tilde t_{\eps,b}(u,u)=\int_{\Tilde V_{\eps,b}} |\nabla u|^2\dd x- \int_{\partial_0 \Tilde V_{\eps,b}} u^2\, \dd\sigma,
	\quad
	D(\Tilde t_{\eps,b})=\Hat H^1_0(\Tilde V_{\eps,b}).
\]

\begin{lemma}\label{lem11}
For any $\eps_0>0$ and $b>0$ there exists $c>0$ such that
\[
\inf\spec \Tilde T_{\eps,b}\ge -\frac{c}{\eps}
\text{ for all $\eps\in(0,\eps_0)$.}
\]
\end{lemma}

\begin{proof}
Let $u\in C^\infty_{(b,\infty)}(\Bar\Omega_\eps)$, then due to Lemma~\ref{lem6a}
one has
\begin{multline}
	\label{ekk}
	\Tilde t_{\eps,b}(u,u)\ge\int_{\Tilde V_{\eps,b}} |\partial_{x_1} u|^2\dd x
	+\int_b^\infty \Big\{
	\int_{\eps x_1\omega} |\nabla_{x'} u (x_1,x')|^2\dd x'\\-(\eps x_1)^{n-1}\sqrt{1+R^2\eps^2}\int_{\partial \omega} u(x_1,\eps x_1 t)^2\dd\tau(t)
	\Big\}\dd x_1.
\end{multline}
We have
\begin{align*}
\int_{\eps x_1\omega} |\nabla_{x'} u (x_1,x')|^2\dd x'&=(\eps x_1)^n \int_{\omega} |(\nabla_{x'} u) (x_1,\eps x_1 t)|^2\dd t\\
&=(\eps x_1)^{n-2} \int_{\omega} |\nabla_{t} u (x_1,\eps x_1 t)|^2\dd t 
\end{align*}
and then 
\begin{align*}
\int_{\eps x_1\omega} &|\nabla_{x'} u (x_1,x')|^2\dd x'-(\eps x_1)^{n-1}\sqrt{1+R^2\eps^2}\int_{\partial \omega} u(x_1,\eps x_1 t)^2\dd\tau(t)\\
&=
(\eps x_1)^{n-2}\bigg[
\int_{\omega} |\nabla_{t} u (x_1,\eps x_1 t)|^2\dd t-\eps x_1\sqrt{1+R^2\eps^2}\int_{\partial \omega} u(x_1,\eps x_1 t)^2\dd\tau(t)
\bigg]\\
&=(\eps x_1)^{n-2}b_{\eps x_1\sqrt{1+R^2\eps^2}}\big(u(x_1,\eps x_1\cdot),u(x_1,\eps x_1\cdot)\big)\\
&\ge (\eps x_1)^{n-2} E_1 (B_{\eps x_1\sqrt{1+R^2\eps^2}}) \int_\omega u(x_1,\eps x_1 t)^2\dd t\\
&=\dfrac{1}{(\eps x_1)^2} E_1 (B_{\eps x_1\sqrt{1+R^2\eps^2}}) \int_{\eps x_1\omega}u(x_1,x')^2\dd x'.
\end{align*}
The substitution into \eqref{ekk} gives
\begin{equation}
\begin{aligned}
\Tilde t_{\eps,b}(u,u)&\ge \int_b^\infty \dfrac{E_1(B_{\eps x_1\sqrt{1+R^2\eps^2}})}{(\eps x_1)^2}\int_{\eps x_1\omega } u^2\dd x'\dd x_1\\
&\ge \inf_{x_1>b} \dfrac{E_1(B_{\eps x_1\sqrt{1+R^2\eps^2}})}{(\eps x_1)^2}
\int_b^\infty \int_{\eps x_1 \omega} u^2\dd x'\dd x_1\\
&\equiv\inf_{x_1>b} \dfrac{E_1(B_{\eps x_1\sqrt{1+R^2\eps^2}})}{(\eps x_1)^2}\|u\|^2_{L^2(\Tilde V_{\eps,b})}.
\end{aligned}
   \label{eqlok5}
\end{equation}
By Lemma \ref{lem-1}(c) there exists $c_0>0$ such that $E_1(B_{r})\ge -N_\omega r-c_0r^2$ for all $r>0$. Hence, for any $x_1>b$ we have
\begin{align*}
\dfrac{E_1(B_{\eps x_1\sqrt{1+R^2\eps^2}})}{(\eps x_1)^2}&\ge \dfrac{-N_\omega\eps x_1\sqrt{1+R^2\eps^2} -c_0 \eps^2 x_1^2(1+R^2\eps^2)}{(\eps x_1)^2}\\
&\equiv-\dfrac{N_\omega \sqrt{1+R^2\eps^2}}{\eps x_1}-c_0(1+R^2\eps^2)\\
&\ge -\dfrac{N_\omega \sqrt{1+R^2\eps_0^2}}{ b\eps}-c_0(1+R^2\eps_0^2)\ge-\dfrac{c}{\eps}
\end{align*}
with $c:=\big(N_\omega\sqrt{1+R^2\eps_0^2}+b\eps_0 c_0(1+R^2\eps_0^2)\big)/b$,
and the substitution into \eqref{eqlok5} gives the result.
\end{proof}

\begin{lemma}\label{lem12}
Let $j\in\NN$. Then there exist $K>0$ and $\eps_{0}>0$ such that for all $\eps\in(0,\eps_{0})$ there holds
$\Lambda_{j}(Q_\eps)\ge \Lambda_j (T_\eps)-K$.
\end{lemma}

\begin{proof} The argument uses the same idea as in Lemma~\ref{lem7}.
Let $\chi_{1},\chi_{2}\in C^{\infty}(0,\infty)$ with $0\le \chi_{1},\chi_{2}\le 1$
and $\chi_{1}^{2}+\chi_{2}^{2}=1$, such that $\chi_{1}(s)=0$ for $s\geq\frac{3a}{4}$ and $\chi_{2}(s)=0$ for $s\leq \frac{a}{2}$.
We set $K:=\lVert\chi_{1}'\rVert_{\infty}^{2}+\lVert\chi_{2}'\rVert_{\infty}^{2}$
and define functions $\rho_j:(x_1,x')\mapsto \chi_j(x_1)$, then $\rho_1^2+\rho_2^2=1$
and $\|\nabla\rho_1\|_\infty^2+\|\nabla\rho_2\|_\infty^2=K$. It is convenient to 
denote $b:=\frac{a}{4}$. For any $u\in C^\infty_{(0,\infty)}(\Bar\Omega_\eps)$ one has
\begin{align*}
\int_{\Omega_\varepsilon}|\nabla u|^{2}\dd x&=\int_{\Omega_{\varepsilon}}|\nabla(\rho_1 u)|^{2}\dd x+\int_{\Omega_{\varepsilon}}|\nabla(\rho_2 u)|^{2}\dd x
-\int_{\Omega_{\varepsilon}}u^{2}(|\nabla \rho_1|^{2}+|\nabla\rho_2|^{2})\dd x\\
&\ge \int_{\Omega_{\varepsilon}}|\nabla(\rho_1 u)|^{2}\dd x+\int_{\Omega_{\varepsilon}}|\nabla(\rho_2 u)|^{2}\dd x
-K \int_{\Omega_{\varepsilon}}u^{2}\dd x.
\end{align*}
As $\rho_1 u$ vanishes for $x_1>\frac{3 a}{4}$ and $\rho_2 u$ vanishes for $x_1<\frac{a}{2}$,
one can rewrite the last inequality as
\[
\int_{\Omega_\varepsilon}|\nabla u|^{2}\dd x+K \int_{\Omega_{\varepsilon}}u^{2}\dd x
\ge \int_{V_\eps}|\nabla(\rho_1 u)|^{2}\dd x+\int_{\Tilde V_{\eps,b}}|\nabla(\rho_2 u)|^{2}\dd x.
\]
Also remark that $\rho_1u\in \Hat H^1_0(V_\eps)$ and $\rho_2u\in \Hat H^1_0(\Tilde V_{\eps,b})$, and 
\begin{align*}
\int_{\partial \Omega_\eps}|u|^2\dd\sigma&=\int_{\partial\Omega_\eps}|\rho_1u|^2\dd\sigma +\int_{\partial\Omega_\eps}|\rho_2 u|^2\dd\sigma=\int_{\partial_0 V_\varepsilon}|\rho_1 u|^{2}\dd\sigma +\int_{\partial_0 \Tilde V_{\varepsilon,b}}|\rho_2 u|^2\dd\sigma,\\
\lVert u\rVert^{2}_{L^2(\Omega_\eps)}&=\int_{\Omega_\eps}|\rho_1 u|^2\dd x+\int_{\Omega_\eps}|\rho_2 u|^{2}\dd x\\
&=\int_{V_{\varepsilon}}|\rho_1 u|^{2}\dd x+\int_{\Tilde V_{\varepsilon,b}}|\rho_2u|^{2}\dd x
=\lVert \rho_{1} u\rVert^{2}_{L^{2}(V_\eps)}+\lVert \rho_{2}u\rVert^{2}_{L^2(\Tilde V_{\varepsilon,b})}.
\end{align*}
Substituting these computations into the expression for $q_\eps(u,u)$ we obtain
\[
q_\eps(u,u)+K \|u\|^2_{L^2(\Omega_\eps)}\ge t_\eps(\rho_1 u,\rho_1 u)+\Tilde t_{\eps,b}(\rho_2 u,\rho_2 u)
\text{ for any $u\in C^\infty_{(0,\infty)}(\Bar\Omega_\eps)$,}
\]
and it follows, using the min-max principle, that for any $j\in\NN$,
\begin{equation}
	\label{eqlok7}
	\begin{aligned}
\Lambda_{j}(Q_{\varepsilon})+K&
=\inf_{\substack{S\subset C^\infty_{(0,\infty)}(\Bar\Omega_\eps)\\ \dim S=j} }\sup_{\substack{u\in S\\ u\neq 0}}\frac{ q_\eps(u,u)+K \|u\|^2_{L^2(\Omega_\eps)} }{\lVert u\rVert^{2}_{L^{2}(\Omega_\eps)}}\\
&\geq\inf_{\substack{S\subset C^\infty_{(0,\infty)}(\Bar\Omega_\eps)\\ \dim S=j} }\sup_{\substack{u\in S\\ u\neq 0}}\frac{t_{\varepsilon}(\rho_1 u,\rho_1 u)+\Tilde t_{\varepsilon,b}(\rho_2 u,\rho_2 u)}{\lVert \rho_{1} u\rVert^{2}_{L^{2}(V_{\varepsilon})}+\lVert \rho_{2}u\rVert^{2}_{L^{2}(\Tilde{V}_{\varepsilon,b})}}
\\
&\geq \inf_{\substack{S\subset D(t_{\varepsilon})\oplus D(\Tilde t_{\varepsilon,b})\\ \dim S=j} }\sup_{\substack{(u_1,u_2)\in S\\ (u_1,u_2)\neq 0}}\frac{t_{\varepsilon}(u_{1},u_{1})+\Tilde t_{\varepsilon,b}(u_{2},u_{2})}{ \lVert u_1\rVert^{2}_{L^{2}(V_{\varepsilon})}+\lVert u_2\rVert^{2}_{L^{2}(\Tilde{V}_{\varepsilon,b})} }=\Lambda_{j}(T_{\varepsilon}\oplus \Tilde T_{\varepsilon,b}).
\end{aligned}
\end{equation}

Now let $j\in\N$ be fixed. As $\eps\to 0^+$, by Proposition \ref{prop8} we have $\Lambda_j(T_\eps)\sim -c\eps^{-2}$ with some $c>0$, and by Lemma~\ref{lem11} we have the bound $\inf\spec \Tilde T_{\eps,b}\ge -\Tilde c\eps^{-1}$ with some $\Tilde c>0$. So for all sufficiently small $\eps>0$ one has $\Lambda_j(T_\eps)<\inf\spec \Tilde T_{\eps,b}$, which implies $\Lambda_{j}(T_{\varepsilon}\oplus \Tilde T_{\varepsilon,b})=\Lambda_j(T_\eps)$. The substitution into \eqref{eqlok7} finishes the proof.
\end{proof}

The following assertion together with the asymptotics of $\Lambda_j(T_\eps)$ from Proposition~\ref{prop8}
completes our proof of Theorem \ref{thm1}:

\begin{prop}\label{prop9}
	Let $j\in\NN$ be fixed, then:
	\begin{itemize}
		\item one can find some $\eps_j>0$ such that $Q_\eps$ has at least $j$ discrete eigenvalues
	below $\inf\spec_\ess Q_\eps$ for all $\eps\in(0,\eps_j)$,
	\item there holds $E_j(Q_\eps)=\Lambda_j(T_\eps)+O(1)$ as $\eps\to 0^+$.
	\end{itemize}
\end{prop}

\begin{proof}
Let us fix $j\in \NN$.  By combining the upper bound of Lemma~\ref{lem10} and the lower bound of Lemma~\ref{lem12}
we obtain $\Lambda_j(Q_\eps)=\Lambda_j(T_\eps)+O(1)$. By Proposition~\ref{prop8}
we have $\Lambda_{j+1}(T_\eps)-\Lambda_{j}(T_\eps)\to +\infty$ as $\eps\to 0^+$. It follows that
there exists $\eps_j>0$ such that $\Lambda_j(Q_\eps)<\Lambda_{j+1}(Q_\eps)$ for all $\eps\in(0,\eps_j)$,
and then $E_j(Q_\eps)=\Lambda_j(Q_\eps)$ for the same $\eps$ due to the min-max principle.
\end{proof}

\begin{rem}
By sending $j$ to $\infty$ in \eqref{eqlok7} and using the compactness of the resolvent of $T_\eps$ one also shows that
$\inf\spec\nolimits_\ess Q_\eps\ge \inf\spec_\ess \Tilde T_{\eps,b}$.
Then Lemma \ref{lem11} shows that for small $\eps>0$ one estimates $\inf\spec\nolimits_\ess Q_\eps\ge -c\eps^{-1}$
with a fixed $c>0$. For each $j\in\NN$ the difference between $\inf\spec_\ess Q_\eps$ and $E_j(Q_\eps)$ is of order $\eps^{-2}$, so it is standard to show that the respective eigenfunction $u_{j,\eps}$ satisfies an Agmon-type exponential decay at infinity \cite{agmon}. This explains why the analysis of $Q_\eps$ on the complete infinite cone can be reduced to the analysis of an operator on a finite part of the cone. \end{rem}

\appendix

\section*{Appendix}

\section{Proof of Lemma \ref{prop4dens}}\label{appa}

Let us first recall some basic definitions from the theory of Sobolev spaces, as they will be actually used during the proofs. We mostly follow the convention
from the book~\cite{gris}. For an open set $\Omega\subset\R^m$  and $k\in\NN$ the $k$-th Sobolev space $H^k(\Omega)$ is defined
as
\[
H^k(\Omega):=\big\{u\in L^2(\Omega): \partial^\alpha u\in L^2(\Omega) \text{ for all $|\alpha|\le k$}\}
\]
with all derivatives taken in the sense of distributions, and it is a Hilbert space
with respect to the scalar product
\[
\langle u,v\rangle_{H^k(\Omega)}:=\sum_{|\alpha|\le k} \langle \partial^\alpha u,\partial^\alpha v\rangle_{L^2(\Omega)}.
\]
By $C^\infty(\overline{\Omega})$ one denotes the set of functions defined on $\Omega$ which can be extended
to functions in $C^\infty_c(\R^m)$.
One says that an open set $\Omega\subset\R^m$ has $C^k$ (respectively Lipschitz) boundary, if for any $p\in \partial\Omega$ there exist Cartesian coordinates $(y_1,\dots,y_m)$ centered at $p$, a $C^k$ (respectively Lipschitz) function $h$ of $m-1$ variables,
defined on an open neighborhood of $0$ in $\R^{m-1}$ and with $h(0,\dots,0)=0$, and $\eps>0$ such that
\[
\Omega\cap B_\eps(p)=\big\{y=(y_1,\dots,y_m)\in B_\eps(0):\, y_m<h(y_1,\dots,y_{m-1})\big\}.
\]
Most assertions used in the theory of Sobolev spaces (some density and extension results, trace theorems) are usually formulated
for bounded open sets with Lipschitz boundaries. On the other hand, the cone $\Omega_\eps$ has in general \emph{not} even a $C^0$ boundary:
for example, if $n=2$ and $\omega$ is an annulus, $\omega=\{(x_1,x_2):1<x_1^2+x_2^2<4\}$, then
one easily sees that $\Omega_\eps$ cannot be represented as one of the sides of the graph of a continuous function
near the vertex $0$. Moreover, further common assumptions used in the theory of Sobolev spaces (e.g. the segment condition or the cone condition)
fail as well.

We collect some well known facts about $H^k(\Omega)$ in the following proposition:

\begin{prop}\label{prop-sobolev}
Let $\Omega\subset\R^m$ be an open set.
\begin{itemize}
		\item[(A)] The space
		\[
		H^1_\infty(\Omega):=\big\{ u\in H^1(\Omega): u\in C^\infty(\Omega)\cap L^\infty(\Omega),\, \supp u \text{ is bounded\,}\big\}
		\]
		is dense in $H^1(\Omega)$. (Remark that there are no additional assumptions on $\Omega$.)
		\item[(B)] If $\Omega$ has $C^0$ boundary, then $C^\infty(\Bar \Omega)$ is dense in $H^k(\Omega)$ for any $k\in\NN$. 
		\item[(C)] If $\Omega$ is bounded and has Lipschitz boundary, then:
		\begin{itemize}
			\item[(C.1)] for any $k\in\NN$, any function in $H^k(\Omega)$ can be extended to a function in $H^k(\R^m)$.
		    \item[(C.2)] the linear map $C^\infty(\Bar \Omega)\ni u\mapsto u|_{\partial\Omega}\in L^2(\partial\Omega)$ uniquely extends
		    by continuity to a bounded linear map $\gamma_0:H^1(\Omega)\to L^2(\partial\Omega)$. Moreover, for any $\eps>0$
		    there exists $C_\eps>0$ such that
		    \[
		    \int_{\partial\Omega} (\gamma_0 u)^2 d\sigma_{m-1}\le \eps\int_\Omega |\nabla u|^2\dd x+C_\eps \int_\Omega u^2\dd x
		    \]
		    for all $u\in H^1(\Omega)$, where $\sigma_{m-1}$ is the $(m-1)$-dimensional Hausdorff measure.
		   \end{itemize}
	\end{itemize}
\end{prop}

We refer to \cite[Theorem in Sec.~1.4.3]{mp} for (A), to \cite[Theorem 1 in Sec.~1.4.2]{mp} for (B),
to \cite[Thm.~5.2.4]{adams} for (C.1) and to \cite[Theorem 1.5.1.10]{gris} for (C.2). Note that
one usually writes simply $u$ instead of $\gamma_0u$ in the integrals over the boundary.

Now we pass to the discussion of Sobolev spaces on the infinite cones $\Omega_\eps$. We start with several preparation steps.

\begin{lemma}\label{lem-cyl}
	Let $-\infty<a<b<\infty$, then the cylinder $\Omega:=(a,b)\times \omega\subset\R^{n+1}$ has Lipschitz boundary.
\end{lemma}

\begin{proof}
	Let $p\in\partial\Omega$, then the following cases are possible.

	Case 1: $p=(a',p')$ with $a'\in(a,b)$ and $p'\in \partial\omega$. Since $\omega$ has Lipschitz boundary, there exists
	Cartesian coordinates $(y_1,\dots,y_n)$ in $\R^n$ centered at $p'$ and a Lipschitz function $h$ with $h(0)=0$
	such that $\omega$ coincides with $\{y: y_n<h(y_1,\dots,y_{n-1})\}$ near $p'$. Denote $z:=x_1-a'$, then
	$(z,y_1,\dots,y_n)$ are Cartesian coordinates in $\R^{n+1}$ centered at $p$, and $\Omega$ near $p$
	coincides with $\{(z,y): y_n<H(z,y_1,\dots,y_{n-1})\}$ for the function $H(z,y_1,\dots,y_{n-1}):=h(y_1,\dots,y_{n-1})$,
	which is obviously Lipschitz.
	
	Case 2a: $p=(a,x')$ with $x'\in\omega$, then $\Omega$ near $p$ coincides with $\{(z,y): z<0\}$, where $y=(y_1,\dots,y_n)$ are arbitrary Cartesian coordinates in $\R^n$ centered at $x'$ and $z:=a-x_1$: remark that $(z,y_1,\dots,y_n)$ are Cartesian coordinates in $\R^{n+1}$ centered at $p$, and the zero function is obviously Lipschitz. Case 2b: $p=(b,x')$ with $x'\in\omega$ is treated analogously.
	
	Case 3a: $p=(a,p')$ with $p'\in \partial\omega$ (the most difficult one). Since $\omega$ has Lipschitz boundary, there exist Cartesian coordinates $(y_1,\dots,y_n)$ in $\R^n$ centered at $p'$ and a Lipschitz function $h$ with $h(0)=0$
	such that $\omega$ coincides with $\{y: y_n<h(y_1,\dots,y_{n-1})\}$ near $p'$. Remark that $\Omega$ near $p$ is then determined by the two inequalities
	\begin{equation}
		\label{x1y}
		x_1>a, \quad y_n< h(y_1,\dots,y_{n-1}).
	\end{equation}
	In order to bring these conditions
	into the required form we pick $\theta\in(0,\frac{\pi}{2})$ and apply a rotation by the angle $\theta$ around $p$ in the $(x_1,y_n)$-plane. Namely, consider the Cartesian coordinates $(z,y_1,\dots,y_{n-1},w)$
	with the previous $y_1,\dots,y_{n-1}$ and
	\[
	\begin{pmatrix} x_1-a\\ y_n\end{pmatrix}= z \begin{pmatrix} \cos \theta \\ \sin\theta \end{pmatrix} + w\begin{pmatrix} -\sin\theta \\ \cos\theta \end{pmatrix}.
	\]
	Clearly, the new coordinates are centered at $p$, and the above inequalities \eqref{x1y} determining $\Omega$ near $p$ take the form
	\begin{equation}
		w<\frac{\cos\theta}{\sin\theta}\,, \quad w<-\frac{\sin\theta}{\cos\theta}\, z+\frac{1}{\cos\theta}\, h(y_1,\dots,y_{n-1}),
	\end{equation}
	which can be rewritten as
	\[
	w<H(z,y_1,\dots,y_{n-1}):=\min\Big\{\frac{\cos\theta}{\sin\theta}\,z,-\frac{\sin\theta}{\cos\theta}\, z+\frac{1}{\cos\theta}\, h(y_1,\dots,y_{n-1})\Big\}.
	\]
	The function $H$ is Lipschitz (since it is the minimum of two Lipschitz functions), hence, one has a required representation of $\Omega$ near $p$. The case 3b:  $p=(b,p')$ with $p'\in \partial\omega$ is considered analogously.
\end{proof}
The next step is somewhat technical; it shows that each $H^1$-function on a cylinder vanishing identically near the bases
can be suitably approximated by test functions vanishing near the bases.
\begin{lemma}\label{lem-cyl2}
Let $-\infty<a<b<\infty$ and $\Omega:=(a,b)\times \omega\subset\R^{n+1}$. Let
$[c,c']\subset (a,b)$ and $u\in H^1(\Omega)$ be such that $u(x_1,x')=0$ for $x_1\notin [c,c']$.
Let $0<\delta<\min\{c-a,b-c'\}$, then for any $\eps>0$ there exists $\varphi\in C^\infty_c(\R^{n+1})$ such that
 $\|u-\varphi\|_{H^1(\Omega)}<\eps$ and $\varphi(x_1,x')=0$ for all $x_1\notin[c-\delta,c'+\delta]$.
\end{lemma}	

\begin{proof}
By Lemma \ref{lem-cyl} and Proposition~\ref{prop-sobolev} (part C.1) the function $u$ can be extended to a function $v'\in H^1(\R^{n+1})$.
Choose $\chi\in C^\infty_c(\R)$ such that $\chi(s)=1$ for $s\in[c,c']$ and $\supp\chi\in \big[c-\frac{\delta}{2},c'+\frac{\delta}{2}\big]$,
and, in addition, choose $\chi_0\in C^\infty_c(\R^n)$ with $\chi_0=1$ on $\omega$. Then
the function $v:(x_1,x')\mapsto \chi(x_1)\chi_0(x')v'(x_1,x')$ belongs to $H^1(\R^{n+1})$, is an extension of $u$, has compact support, and $v(x_1,x')=0$ for all $x_1\notin \big[c-\frac{\delta}{2},c'+\frac{\delta}{2}\big]$.

Let $\rho\in C^\infty_c(\R^{n+1})$ with
\[
\rho(y)=0 \text{ for } |y|\ge 1, \quad
\int_{\R^{n+1}}\rho(y)\dd y=1,
\]
and for $t>0$ consider the functions $\rho_t:x\mapsto t^{-(n+1)}\rho(t^{-1}x)$.
Then $v_t:=v\ast\rho_t\in C^\infty_c(\R^{n+1})$, where $\ast$ denotes the convolution product,
and  $\|v_t-v\|_{H^1(\R^{n+1})}\to 0$ for $t\to 0^+$. Hence, there exists some $t_0>0$ such that
$\|v_t-v\|_{H^1(\R^{n+1})}<\eps$ for all $t\in(0,t_0)$.

Furthermore, the definition of the convolution product implies the inclusion $\supp v_t\subset \supp v+\Bar B_t(0)$.
In particular, if $t<\frac{\delta}{2}$,
then one has $v_t(x_1,x')=0$ for all $x_1\notin[c-\delta,c'+\delta]$. Now pick any $0<t<\min\big\{t_0,\frac{\delta}{2}\big\}$
and denote $\varphi:=v_t$, then
\[
\|\varphi-u\|_{H^1(\Omega)}=\|v_t-v\|_{H^1(\Omega)}\le \|v_t-v\|_{H^1(\R^{n+1})}<\eps,
\]
so $\varphi$ has all the required properties.	
\end{proof}

Let us finish the proof of Lemma \ref{prop4dens}.

	First remark that $H^1_\infty(\Omega_\eps)$ is dense in $H^1(\Omega_\eps)$ by Proposition \ref{prop-sobolev}(A).
	So we need to show that any function in $H^1_\infty(\Omega_\eps)$ can be approximated by functions from $C^\infty_{(0,\infty)}(\Bar \Omega_\eps)$
	in the $H^1$-norm.
	
	Let $v\in H^1_\infty(\Omega_\eps)$, then there exists some $c\in(0,\infty)$ such that $v(x_1,x')=0$ for $x_1>c$.
	Let $\chi:\R\to \R$ be a $C^\infty$-function with $0\le\chi\le 1$, $\chi(s)=0$ for $s<\frac{1}{2}$, and $\chi(s)=1$ for $s>1$. For $\delta>0$ consider the functions
	\[
	v_\delta:(x_1,x')\mapsto \chi\big(\frac{x_1}{\delta}\big)v(x_1,x').
	\]
	We have
	\[
	\|v_\delta -v\|^2_{L^2(\Omega_\eps)}=\int_{\Omega_\eps} \big| 1 -\chi\big(\frac{x_1}{\delta}\big)^2\big| \,v(x_1,x')^2\dd x
	\le \int_{\Omega_\eps\cap\{x_1<\delta\}} \,v(x_1,x')^2\dd x\stackrel{\delta\to 0^+}{\longrightarrow}0.
	\]
	Furthermore,
	\begin{align*}
		\partial_1 v_\delta(x_1,x')&=\frac{1}{\delta} \chi'\big(\frac{x_1}{\delta}\big) v(x_1,x')+\chi\big(\frac{x_1}{\delta}\big)\partial_1 v(x_1,x'),\\
		\partial_j v_\delta(x_1,x')&=\chi\big(\frac{x_1}{\delta}\big)\partial_j v(x_1,x') \text{ for }j\ge 2.
	\end{align*}
For every $j\ge 2$ one obtains
	\begin{align*}
		\|\partial_j v_\delta -\partial_j v\|^2_{L^2(\Omega_\eps)}&=\int_{\Omega} \big| 1 -\chi\big(\frac{x_1}{\delta}\big)^2\big| \,\big|\partial_j v(x_1,x')\big|^2\dd x\\
		&\le \int_{\Omega_\eps\cap\{x_1<\delta\}} \big|\partial_j v(x_1,x')\big|^2\dd x\stackrel{\delta\to 0^+}{\longrightarrow}0.
	\end{align*}
	In addition, using $(x+y)^2\le 2(x^2+y^2)$ we estimate
	\begin{equation}
		\begin{aligned}
			\|\partial_1 v_\delta -\partial_1 v\|^2_{L^2(\Omega_\eps)}&\le 2\int_{\Omega_\eps} \big| 1 -\chi\big(\frac{x_1}{\delta}\big)^2\big| \,\big|\partial_1 v(x_1,x')\big|^2\dd x\\
			&\quad 	+\frac{2}{\delta^2}\int_{\Omega_\eps} \chi'\big(\frac{x_1}{\delta}\big)^2 v(x_1,x')^2\dd x	\\
			&\le \int_{\Omega_\eps\cap\{x_1<\delta\}} \big|\partial_1 v(x_1,x')\big|^2\dd x
			+\dfrac{2}{\delta^2}\, \|\chi'\|_\infty^2 \|v\|^2_\infty \int_{\Omega_\eps\cap\{x_1<\delta\}}\dd x.
		\end{aligned}
		\label{temp6}
	\end{equation}
	The first summand on the right-hand converges to $0$ as $\delta\to 0^+$ since $\partial_1 v\in L^2(\Omega_\eps)$. We further note that
	\[
	\int_{\Omega_\eps\cap\{x_1<\delta\}}\dd x=\int_0^\delta \int_{\eps x_1 \omega} \dd x'\dd x_1=\eps^n \vol_n \omega \int_0^\delta x_1^n\dd x_1 =\eps^n \vol_n \omega \,\frac{\delta^{n+1}}{n+1},
	\]
	and the second summand on the right-hand side of \eqref{temp6} is estimated from above by $\frac{2}{n+1}\,\eps^n \|\chi'\|_\infty^2 \|v\|^2_\infty \vol_n\omega \delta^{n-1}$, which converges to $0$
	as $\delta\to 0^+$ due to $n\ge 2$. We have proved that $v_\delta$ converges to $v$ in $H^1(\Omega_\eps)$ as $\delta\to 0^+$. Remark that $v_\delta(x_1,x')=0$ for $x_1\notin[\delta/2,c]$, therefore, the above constructions show that
	the subspace
	\begin{multline*}
		\cD:=\big\{ u\in H^1(\Omega_\eps)\cap C^\infty(\Omega_\eps)\cap L^\infty(\Omega_\eps): \\
		\text{ $\exists\, [b,c]\subset (0,\infty)$ such that }
		\text{$u(x_1,x')=0$ for $x_1\notin [b,c]$} \big\},
	\end{multline*}
    is dense in $H^1(\Omega_\eps)$. Now it remains to check that each function in $\cD$ can be approximated by functions from $ C^\infty_{(0,\infty)}(\Bar \Omega_\eps)$ in $H^1(\Omega_\eps)$.

	Let $u\in\cD$ and $[b,c]\subset(0,\infty)$ such that $u(x_1,x')=0$ for $x_1\notin[b,c]$. The map
	$X:(0,\infty)\times \R^n\ni (s,t)\mapsto (s,\eps st)\in (0,\infty)\times\R^n$
	is a diffeomorphism with $X\big((0,\infty)\times \omega\big)=\Omega_\eps$.
	Pick an arbitrary $\delta\in(0,\frac{b}{2})$ and denote $\Omega':=(b-2\delta,c+2\delta)\times \omega$,
	then the function $u_X:=u\circ X$ belongs to $H^1(\Omega')$.
	
	Let $\mu>0$, then by Lemma~\ref{lem-cyl2} one can find $\varphi^\mu_X\in C^\infty_c(\R^{n+1})$ with
	\[
	\|u_X-\varphi^\mu_X\|_{H^1(\Omega')}<\mu, \quad \varphi^\mu_X(x_1,x')=0 \text{ for all } x_1\notin[b-\delta,c+\delta].
	\]
	Then the functions
	\[
	\varphi^\mu:\R^{n+1}\ni x\mapsto \begin{cases}
		\varphi_X^\mu \big(X^{-1}(x)\big), & x_1>0,\\
		0, & \text{otherwise}
	\end{cases}
	\]
	belong to $C^\infty_c(\R^{n+1})$ and $\varphi^\mu(x_1,x')=0$ for all $x_1\notin[b-\delta,c+\delta]$, i.e. the restriction of $\varphi^\mu$
	to $\Omega_\eps$ belongs to $ C^\infty_{(0,\infty)}(\Bar \Omega_\eps)$.

	The supports of $u$ and $\varphi^\mu$ are contained in $[b-\delta,c+\delta]\times\R^n$
	and all derivatives of $X$ and $X^{-1}$ are uniformly bounded on the compact sets
	$\overline{\Omega'}$ and $\Bar\Omega_\eps\cap\{x_1\in [b-2\delta,c+2\delta]\}$ respectively.
	Therefore, one can find some $C>0$ such that
	\[
	\|u-\varphi^\mu\|_{H^1(\Omega_\eps)}\equiv
	\|u-\varphi^\mu\|_{H^1(\Omega_\eps \cap \{x_1\in (b-2\delta,c+2\delta)\})}
	\le C\|u_X-\varphi^\mu_X\|_{H^1(\Omega')}
	\] for all $\mu>0$.
	Since $\mu$ can be taken arbitrarily small, this concludes the proof.

\section{Traces and semiboundedness}\label{appb}
Firstly we prove Lemma \ref{lem6a} and then move on to show the wellposedness of our spectral problem, i.e. the lower semiboundedness and closedness of $q_\eps$.

\begin{proof}[Proof of Lemma \ref{lem6a}]
	As usual for the integration over hypersurfaces, it is sufficient to prove the statement for functions supported in images
	of local charts, then this statement is extended to general functions using a partition of unity.	
	
	Let $U\ni z=(z_1,\dots,z_{n})\mapsto \varphi(z)$ be a local chart on $\partial\omega$, then
	\[
	\Phi:(0,\infty)\times U\ni (s,z)\mapsto \big(s,\eps s\varphi(z)\big)\equiv X\big(s,\varphi(z)\big)\in \partial \Omega_\eps
	\]
	is a local chart on $\partial \Omega_\eps$. 	If $v|_{\partial \Omega_\eps}$ is supported in the image
	of $\Phi$, then
	\begin{equation}
		\label{eq-veps}
		\int_{\partial \Omega_\eps} |v|\dd\sigma = \int_0^\infty \int_{U} \left|v\big( \Phi(s,z)\big)\right|\, g_\Phi(s,z)\dd z \dd s,\quad
		g_\Phi:=\sqrt{\det (D\Phi^T D\Phi)}.
	\end{equation}
	We compute
	\begin{align*}
		(D\Phi^T D\Phi)(s,z)&=\begin{pmatrix}
			1+\eps^2 |\varphi(z)|^2 & \eps^2 s F(z)^T\\[\smallskipamount]
			\eps^2 s F(z) & \eps^2 s^2 G_\varphi(z)
		\end{pmatrix}, \quad G_\varphi:=D\varphi^T D\varphi,\\
		F(z)&:=\begin{pmatrix}
			\big\langle \varphi(z),\partial_1 \varphi(z)\big\rangle_{\R^n}\\[\smallskipamount]
			\vdots\\
			\big\langle \varphi(z),\partial_{n-1} \varphi(z)\big\rangle_{\R^n}
		\end{pmatrix}\equiv \dfrac{1}{2}\nabla_z \big|\varphi(z)\big|^2\equiv \big|\varphi(z)\big| \nabla_z \big|\varphi(z)\big|.
	\end{align*}
	The matrix $G_\varphi$ is invertible a.e. (as $\varphi$ is a local chart), therefore, using
	well-known formulas for the determinants of block matrices (see e.g. \cite{sylv}) we obtain
	\begin{align*}
		g_\Phi&(s,z)^2\equiv \det (D\Phi^T D\Phi)\\
		&=\left(1+ \eps^2 |\varphi(z)|^2 - \Big\langle \eps^2 s F(z),  \big( \eps^2 s^2 G_\varphi(z)\big)^{-1}
		\eps^2 s F(z)\Big\rangle\right) \det  \big( \eps^2 s^2 G_\varphi(z)\big)\\
		&=\eps^{2(n-1)}s^{2(n-1)}\Big[1+\eps^2 |\varphi(z)|^2 \Big( 1-\Big\langle \nabla_z |\varphi(z)|, G_\varphi(z)^{-1} \nabla_z |\varphi(z)| \Big\rangle\Big)\Big] \det G_\varphi(z).
	\end{align*}
	Consider the function $r: \partial\omega\ni t\mapsto |t|\in \R$, then
	\[
	\big\langle \nabla_z |\varphi(z)|, G_\varphi(z)^{-1} \nabla_z |\varphi(z)| \big\rangle
	= |\nabla^{\partial\omega} r|^2 \big(\varphi(z)\big),
	\quad
	|\varphi(z)|^2=r^2\big(\varphi(z)\big),
	\]
	with $\nabla^{\partial\omega} r$ being the tangential gradient of $r$ along $\partial\omega$.
	Therefore,
	\begin{gather*}
		g_\Phi(s,z)=\eps^{n-1}s^{n-1}\sqrt{1+\eps^2 \rho(\varphi(z))}
		g_\varphi(z), \quad g_\varphi(z):=\sqrt{\det G_\varphi(z)},\\
		\rho:=r^2\big(1-|\nabla^{\partial\omega} r|^2\big)\equiv r^2\big(|\nabla^{\R^n} r|^2-|\nabla^{\partial\omega} r|^2\big)\equiv r^2 |\partial_\nu r|^2
	\end{gather*}
	with $\partial_\nu$ being the normal derivative. Due to $|\partial_\nu r|\le 1$ we have $0\le \rho\le R^2$
	with $R$ from \eqref{rrr}. By \eqref{eq-veps} we obtain
	\begin{multline}
		\label{eq-uv3}
		\eps^{n-1}\int_0^\infty \int_{U} s^{n-1}\left|v\big( \Phi(s,z)\big)\right| g_\varphi(z)\dd z \dd s\\
		\le
		\int_{\partial \Omega_\eps} |v|\dd\sigma
		\le
		\sqrt{1+R^2\eps^2}\eps^{n-1}\int_0^\infty \int_{U} s^{n-1}\left|v\big( \Phi(s,z)\big)\right| g_\varphi(z)\dd z \dd s.
	\end{multline}
	Using the definition of $\Phi$ we obtain $v\big( \Phi(s,z)\big)=u\big(s,\varphi(z)\big)$ and
	\begin{align*}
		\int_0^\infty \int_{U} s^{n-1}\left|v\big( \Phi(s,z)\big)\right| g_\varphi(z)\dd z \dd s&
		=\int_0^\infty \int_{U} s^{n-1} \left|u( s, \varphi(z))\right| g_\varphi(z)\dd z \dd s\\
		&\equiv \int_0^\infty \int_{\partial\omega} s^{n-1}|u( s, t)| \dd\tau(t) \dd s,	
	\end{align*}
	and the substitution into \eqref{eq-uv3} gives the sought estimate.
	
	The above computations are classical for the case of smooth $\partial\omega$. In our case $\partial\omega$ is only a Lipschitz manifold, but the formulas are still valid a.e.: we refer to~\cite{rosen} for a detailed discussion. 	
\end{proof}

Recall that the subsets $C^\infty_{(0,\infty)}(\Bar\Omega_\eps)$ were defined in \eqref{cinfi}.
The restriction of each function from $C^\infty_{(0,\infty)}(\Bar\Omega_\eps)$ to $\partial\Omega_\eps$
is a continuous function with compact support, hence it is square integrable.

\begin{prop}\label{propb1}
Let $\eps>0$ be fixed. The linear map
\[
\gamma_0: C^\infty_{(0,\infty)}(\Bar\Omega_\eps)\to L^2(\partial\Omega_\eps),
\quad
\gamma_0 u:= u|_{\partial\Omega_\eps},
\]
extends uniquely to a bounded linear map from $H^1(\Omega_\eps)$ to $L^2(\partial\Omega_\eps)$.
Moreover, for any $\delta>0$ there exists $C_\delta>0$ such that
\[
\|\gamma_0 u\|^2_{L^2(\partial\Omega_\eps)}\le \delta \|\nabla u\|^2_{L^2(\Omega_\eps)}+C_\delta \|u\|^2_{L^2(\Omega_\eps)}
\text{ for any } u\in H^1(\Omega_\eps).
\]
\end{prop}

\begin{proof}
It is sufficient to consider $\eps=1$ (as general values of $\eps$ can be absorbed by taking $\eps\omega$ instead of $\omega$).
Since $C^\infty_{(0,\infty)}(\Bar\Omega_1)$ is dense in $H^1(\Omega_1)$ by Proposition~\ref{prop4dens},
it is sufficient to show that for any $\delta>0$ there exists $C_\delta>0$ such that for any 
$u\in C^\infty_{(0,\infty)}(\Bar\Omega_1)$ there holds
\begin{equation}
	  \label{equuu}
\int_{\partial\Omega_1} u^2\dd\sigma \le  \delta
\int_{\Omega_1} |\nabla u|^2\dd x+ C_\delta\int_{\Omega_1} u^2\dd x.
\end{equation}

We use the spectral analysis of the operators $B_r$ from Subsection~\ref{sec-1d}. By Lemma~\ref{lem-1}
one can find a constant $c>0$ such that 
\[
\int_\omega |\nabla v|^2\dd t-r\int_{\partial \omega} v^2\dd\tau\ge -(N_\omega r+cr^2)\int_\omega v^2\dd t \text{ for all }
v\in H^1(\omega), \ r>0.
\]
and the inequality can be rewritten as
\begin{equation}
	\label{uom1}
\int_{\partial \omega} v^2\dd\tau\le \dfrac{1}{r}\int_\omega |\nabla v|^2\dd t
+(N_\omega+cr)\int_\omega v^2\dd t
\text{ for all }
v\in H^1(\omega), \ r>0.
\end{equation}

Recall that by Lemma~\ref{lem6a} we have for any $u\in C^\infty_{(0,\infty)}(\Bar\Omega_1)$:
\[
\int_{\partial \Omega_1} u^2\dd\sigma\le
\sqrt{1+R^2}\int_0^\infty s^{n-1}\int_{\partial\omega} u(s,s t)^2 \dd\tau(t) \dd s.
\]
We are going to control the integral over $\partial\omega$  using~\eqref{uom1} with $v:t\mapsto u(s,st)$
and $r=r(s)$, which gives
\begin{equation}
	\label{eqiii}
	\begin{aligned}
\int_{\partial \Omega_1} u^2\dd\sigma&\le\sqrt{1+R^2}(I_1+I_2),\\
I_1&:=\int_0^\infty \frac{s^{n-1}}{r(s)} \int_\omega |\nabla_t u(s, s t)|^2\dd t\dd s,\\
I_2&:=\int_0^\infty s^{n-1}\big(N_\omega+cr(s)\big) \int_\omega u(s,s t)^2\dd t\dd s.
\end{aligned}
\end{equation}
Now we remark that
\begin{align*}
I_1&
	=
	\int_0^\infty \frac{s^{n+1}}{r(s)} \int_\omega |(\nabla_{x'} u)(s,s t)|^2\dd t\dd s\\
	&=\int_0^\infty \frac{s}{r(s)}\int_{s\omega} |(\nabla_{x'} u)(s,x')|^2\dd x'\dd s\le \int_0^\infty \frac{s}{r(s)}\int_{s\omega} |\nabla u(s,x')|^2\dd x'\dd s.
\end{align*}
Taking $r(s):=\mu s$ with a constant $\mu>0$ to be chosen later we obtain
\[
I_1\le \frac{1}{\mu}\int_0^\infty \int_{s\omega} |\nabla u(s,x')|^2\dd x'\dd s
=\frac{1}{\mu}\int_{\Omega_1} |\nabla u|^2\dd x.
\]
For the same choice of $r(s)$ one has
\begin{align*}
I_2&=\int_0^\infty s^{n-1}(N_\omega+c\mu s) \int_{\omega}u(s, s t)^2\dd t\dd s\\
&=\underbrace{N_\omega \int_0^\infty s^{n-1}\int_{\omega}u(s,s t)^2\dd t\dd s}_{=:J_1}
+c\mu \underbrace{\int_0^\infty s^n\int_{\omega}u(s, s t)^2\dd t\dd s}_{=:J_2}.
\end{align*}
The second term is easy to evaluate:
\begin{align*}
J_2&=\int_0^\infty \int_{s\omega}u(s,x')^2\dd x'\dd s=\int_{\Omega_1} u^2\dd x.
\end{align*}
The term $J_1$ requires a bit more work. We rewrite
\begin{equation}
	 \label{eqj1}
J_1=\int_\omega \int_0^\infty \frac{N_\omega}{s}f_t(s)^2\dd s \dd t\quad \text{with the function}\quad f_t:s\mapsto s^{\frac{n}{2}}u(s,st).
\end{equation}
For each fixed $t$ one has $f_t\in C^\infty_c(0,\infty)$. Using the spectral analysis of Subsection~\ref{compop} (consider the first eigenvalue of $A_{1/\mu}$ with $n=2$) we have
\[
\int_0^\infty \Big[f_t'(s)^2 -\frac{\mu N_\omega}{s}f_t(s)^2\Big]\dd s\ge -\frac{\mu^2 N_\omega^2}{4}\int_0^\infty f_t(s)^2\dd s,
\]
which we rewrite as
\begin{equation}
	\label{eqft}
\int_0^\infty\frac{N_\omega}{s}f_t(s)^2\dd s\le \frac{1}{\mu}\int_0^\infty f'_t(s)^2\dd s+\frac{\mu N_\omega^2}{4}\int_0^\infty f_t(s)^2\dd s.
\end{equation}
We have
\begin{align*}
f_t'(s)^2&=\big(\tfrac{n}{2} s^{\frac{n}{2}-1} u(s,st) +s^{\frac{n}{2}}\partial_s u(s,st)\big)^2\\
&=\tfrac{n^2}{4} s^{n-2} u(s,st)^2 +n s^{n-1} u(s,st)\partial_s u(s,st) + s^n \big|\partial_s u(s,st)\big|^2.
\end{align*}
Using
\begin{align*}
	\int_0^\infty n s^{n-1} u(s,st)\partial_s u(s,st)\dd s&= \frac{n}{2} \int_0^\infty s^{n-1} \partial_s\big(u(s,st)^2\big)\dd s\\
	&=- \frac{n}{2}\,(n-1)\int_0^\infty s^{n-2} u(s,st)^2\dd s
\end{align*}
we arrive at
\[
\int_0^\infty f_t'(s)^2\dd s=\Big[ \frac{n^2}{4}- \frac{n}{2}\,(n-1)\Big]\int_0^\infty s^{n-2}u(s,st)^2\dd s
+\int_0^\infty s^n \big|\partial_s u(s,st)\big|^2\dd s.
\]
Using $n\ge 2$ one obtains
\[
\frac{n^2}{4}- \frac{n}{2}\,(n-1)=\frac{n}{4} \big( n-2(n-1)\big)=\frac{n}{4}\,(2-n)\le 0,
\]
which gives
\[
\int_0^\infty f_t'(s)^2\dd s\le \int_0^\infty s^n \big|\partial_s u(s,st)\big|^2\dd s.
\]
We compute (with the same $R:=\sup_{t\in \omega}|t|$ as above)
\begin{align*}
\big|\partial_s u(s,st)\big|^2&=\big|\partial_{x_1} u(s,st)+t\cdot \nabla_{x'}u(s,st)\big|^2\le 2 |\partial_{x_1} u(s,st)|^2+ 2\big|t\cdot \nabla_{x'}u(s,st)\big|^2\\
&\le 2 |\partial_{x_1} u(s,st)|^2+ 2R^2\big|\nabla_{x'}u(s,st)\big|^2\le 2(1+R^2) |(\nabla u)(s,st)|^2,
\end{align*}
which results in
\[
\int_0^\infty f_t'(s)^2\dd s\le 2(1+R^2)\int_0^\infty s^n |(\nabla u)(s,st)|^2\dd s.
\]
The substitution into \eqref{eqft} gives
\[
\int_0^\infty\frac{N_\omega}{s}f_t(s)^2\dd s\le \frac{2(1+R^2)}{\mu}\int_0^\infty s^n |(\nabla u)(s,st)|^2\dd s+ 
\frac{\mu N_\omega^2}{4}\int_0^\infty s^n  u(s,st)^2\dd s,
\]
and using \eqref{eqj1} one obtains
\begin{align*}
	J_1&\le \frac{2(1+R^2)}{\mu}\int_0^\infty \int_\omega s^n |(\nabla u)(s,st)|^2\dd s \dd t +\frac{\mu N_\omega^2}{4}\int_0^\infty \int_\omega  s^n u(s,st)^2\dd t \dd s\\
	&=\frac{2(1+R^2)}{\mu}\int_0^\infty \int_{s\omega} |\nabla u(s,x')|^2\dd x' \dd t +\frac{\mu N_\omega^2}{4}\int_0^\infty \int_{s\omega}   u(s,x')^2\dd x' \dd s\\
	&=\frac{2(1+R^2)}{\mu}\int_{\Omega_1}|\nabla u|^2\dd x+\frac{\mu N_\omega^2}{4}\int_{\Omega_1}u^2\dd x.
\end{align*}
Using the above estimates for $J_1$ and $J_2$ one obtains:
\[
I_2\le 
\frac{2(1+R^2)}{\mu}\int_{\Omega_1}|\nabla u|^2\dd x+\mu \big(\frac{N_\omega^2}{4}+c\big)\int_{\Omega_1}u^2\dd x,
\]
and the substitution into \eqref{eqiii} gives
\[
\int_{\partial \Omega_1} u^2\dd\sigma\le \sqrt{1+R^2}\dfrac{2(1+R^2)+1}{\mu}\int_{\Omega_1}|\nabla u|^2\dd x
+\sqrt{1+R^2}\,\mu \big(\frac{N_\omega^2}{4}+c\big)\int_{\Omega_1}u^2\dd x.
\]
For any $\delta>0$ one can take $\mu$ sufficiently large, such that the coefficient in front of the first integral
becomes smaller than $\delta$, and this proves the required inequality \eqref{equuu}.
\end{proof}

As an easy corollary we obtain that our spectral problem is well-posed:

\begin{cor}\label{cor-trace}
The bilinear form $q_\eps$ is semibounded from below and closed for any $\eps>0$.
\end{cor}

\end{document}